\documentclass{amsart}
\usepackage{comment}
\usepackage{mathtools,comment}
\usepackage{mathtools}
\usepackage{etex}
\usepackage[usenames,dvipsnames]{pstricks} 
\usepackage{epsfig}
\usepackage{graphicx,color}
\usepackage{geometry}
\usepackage[all]{xy}
\usepackage{amssymb}
\usepackage{amsmath}
\usepackage{cite}
\usepackage{fullpage}
\xyoption{poly}
\usepackage{url}
\usepackage{hyperref}
\usepackage{cleveref}
\numberwithin{equation}{subsection}
\numberwithin{figure}{subsection}
\usepackage{longtable}
\usepackage{verbatim}
\usepackage{tikz}
\usepackage{tikz-cd}
\usepackage{float}
\usetikzlibrary{arrows.meta,bending}

\usepackage{enumitem}
\usepackage{stmaryrd}
\usepackage{forest}

\usetikzlibrary{decorations.pathmorphing}

\newtheorem{theorem}{Theorem}[section]
\newtheorem*{theorem*}{Theorem}
\newtheorem{lemma}[theorem]{Lemma}
\newtheorem{proposition}[theorem]{Proposition}
\newtheorem{corollary}[theorem]{Corollary}
\newtheorem*{corollary*}{Corollary}
\theoremstyle{definition}

\newtheorem{remark}[theorem]{Remark}

\newtheorem*{question*}{Question}

\newtheorem*{steps*}{Answer/steps}

\newtheorem*{progress*}{Progress}
\newtheorem*{example*}{Example}

\newtheorem*{remark*}{Remark}

\newtheorem*{remarks*}{Remarks}
\newtheorem*{definition*}{Definition}

\newtheorem*{conj*}{Conjecture}

\AddToHook{env/lemma/begin}{\crefalias{theorem}{lemma}}
\AddToHook{env/proposition/begin}{\crefalias{theorem}{proposition}}
\AddToHook{env/corollary/begin}{\crefalias{theorem}{corollary}}
\AddToHook{env/definition/begin}{\crefalias{theorem}{definition}}
\AddToHook{env/remark/begin}{\crefalias{theorem}{remark}}
\AddToHook{env/example/begin}{\crefalias{theorem}{example}}

\crefname{lemma}{Lemma}{Lemmas}
\crefname{theorem}{Theorem}{Theorems}
\crefname{proposition}{Proposition}{Propositions}
\crefname{corollary}{Corollary}{Corollaries}
\crefname{definition}{Definition}{Definitions}
\crefname{remark}{Remark}{Remarks}
\crefname{example}{Example}{Examples}

\usepackage{calrsfs}

\usepackage[T1]{fontenc}
\usepackage{textcomp}
\usepackage{times}
\usepackage[scaled=0.92]{helvet}

\usepackage[all]{xy}

\usepackage{forest}
\usepackage{multirow,qtree}

\forestset{
	my tree/.style={
		for tree={
			parent anchor=north,
			child anchor=south,
			grow=90,
		},
		nice empty nodes up
	},
	nice empty nodes up/.style={
		for tree={calign=center},
		delay={where content={}{shape=coordinate,for parent={for children={anchor=south}}}{}}
	}
}

\newcommand{\QQ}{\mathbb{Q}}

\newcommand{\ZZ}{\mathbb{Z}}

\newcommand{\FF}{\mathbb{F}}
\newcommand{\PP}{\mathbb{P}}

\newcommand{\cR}{\mathcal{R}}

\newcommand{\into}{\hookrightarrow}

\newcommand{\ga}{G^{\text{arith}}}

\newcommand{\ggeo}{G^{\text{geom}}}

\newcommand\restr[1]{\raisebox{-.5ex}{$|$}_{#1}}

\DeclareMathOperator{\Res}{Res}

\DeclareMathOperator{\GL}{GL}

\DeclareMathOperator{\Aut}{Aut}

\DeclareMathOperator{\Gal}{Gal}
\DeclareMathOperator{\sgn}{sgn}

\DeclareMathOperator{\Spec}{Spec}

\DeclareMathOperator{\id}{id}

\newcommand{\ilim}{\mathop{\varprojlim}\limits}

\newcommand{\llang}{\langle \langle}
\newcommand{\rrang}{\rangle \rangle}

\title{Arboreal Galois Groups of a PCF Map with Strictly Pre-periodic Critical Points}

\author{\"{O}zlem Ejder}
\address{\"{O}zlem Ejder, Koc University, Faculty of Science, Rumeli Feneri Yolu 34450
Sarıyer, İstanbul, Türkiye}
\email{ozejder@ku.edu.tr}

\author{Zofia Go{\l}aska}
\address{Zofia Go{\l}aska, Faculty of Mathematics and Computer Science, Adam Mickiewicz University, Uniwersytetu Pozna\'nskiego 4, 61-614 Pozna\'n, Poland}
\email{zofia.golaska@amu.edu.pl}

\author{Yasemin Kara}
\address{Yasemin Kara, Bogazici University, Faculty of Science, Mathematics Department, 34342 Bebek, Istanbul, Türkiye}
\email{yasemin.kara@bogazici.edu.tr}

\author{Leonie Nienhaus}
\address{Leonie Nienhaus, Institut für Mathematik, Johann Wolfgang Goethe-Universit{\"a}t, Robert-Mayer-Str. 6-8, 60325 Frankfurt am Main, Germany}
\email{nienhaus@math.uni-frankfurt.de}

\author{\"Ozge \"Ulkem}
\address{\"Ozge \"Ulkem, Academia Sinica, Institute of Mathematics, 6F, Astronomy-Mathematics Building, No. 1, Sec. 4, Roosevelt Road, Da-an, Taipei 106319, Taiwan}
\email{ozgeulkem@as.edu.tw}


\date{}


\begin{document}

\begin{abstract}
    We study the arithmetic and geometric iterated monodromy groups associated to the postcritically finite (PCF) quadratic rational function $f(x)=\frac{2}{(x-1)^2}$ defined over a number field $k$, whose critical points are both strictly pre-periodic. We give explicit recursive descriptions of the topological generators of the geometric iterated monodromy group of $f$  and show that the arithmetic iterated monodromy group has Hausdorff dimension zero.
    We describe an explicit criterion to determine the values $a\in k$ for which the associated arboreal Galois group achieves its maximum possible size. In particular, we show that maximality of the arboreal Galois group can already be verified at level four, which is computationally accessible. 
	We also determine the normalizer of the geometric IMG of $f$ and show that the arithmetic IMG equals the normalizer if $\zeta_8$ is not contained in $k$. Finally, we show that when $k=\QQ$, the constant field contains $\QQ(\mu_{2^{\infty}})$, providing the first full study of a PCF quadratic map with non-abelian constant field.

   \end{abstract}

\date{}
\maketitle
\section{Introduction}\label{intro}

Let $f: \PP_{k}^1\to \PP_{k}^1$ be a rational map of degree $d$ defined over a number field $k$ and fix an algebraic closure $\bar{k}$ of $k$. 
Suppose $a \in k$ is a base point, such that $f^n(x)=a$ has $d^n$ distinct solutions for all $n\geq 1$. The preimages of $a$ under the iterates of $f$ naturally form a regular rooted $d$-ary tree $T_a$: the root is $a$, the vertices at level $n$ are the solutions in $\bar{k}$ to $f^n(x)=a$, and two vertices are connected if $f$ maps one to the other. 
The absolute Galois group $G_k = \Gal(\bar{k}/k)$ acts on this tree, giving rise to a representation
\[ 
\rho_{f,a} : G_k \to \Aut(T_a), 
\]
called the arboreal Galois representation associated to $(f,a)$. The image of this representation, the arboreal Galois group $G_{a}(f)$, 
encodes deep arithmetic information about the iterated preimages of $a$ under $f$. It also encodes information about the forward images of the critical points: a prime $p$ of $k$ can ramify in the tower $k( f^{-n}(a) | n\geq 1)/k$ only if $p$ divides $f^n(c)-a$ for some critical point $c $ of $f$ and and some nonnegative integer $n$. As $a$ varies over $k$, the group $G_{a}(f)$ changes, but it is always contained in a single "generic" group (as long as $a$ is not a forward image of a critical point) - the arithmetic iterated monodromy group $\ga(f)$ - which can be thought of as the arboreal Galois group for a transcendental base point. 

A central question in arithmetic dynamics is: for which $a \in k$ does $G_{a}(f)$ achieve this maximum, i.e. equal to $\ga(f)$? 

These groups were first studied by Odoni in the 1980s \cite{Odoni85, Odoni2, Odoni88} for arithmetical and dynamical purposes, and have since been studied extensively \cite{Stoll, Jones-Manes, arithmeticbasilica, Arborealcubic, BHL, FP20, BEK, BelyiEjder}. A particularly interesting setting is when $f$ is postcritically finite (PCF), i.e., when the critical points of $f$
 have finite orbits. In this case, the tower of extensions (given by the splitting fields of $f^n(x)-a$ for $n\geq 1$) has finitely many ramified primes. It is known that for PCF maps, the arboreal Galois groups have infinite index in $\Aut(T)$, see for example \cite[Theorem 3.1]{Jonessurvey}.
 Most explicit descriptions of the associated Galois groups have been obtained for PCF polynomials \cite{Pinkpolyn, AH-img} . The rational function case remains largely open, with only recent work beginning to address it \cite{EKO, Ejder26, Pinkrational, BDcolliding}.

This article examines the PCF quadratic rational function $f(x)=\frac{2}{(x-1)^2}$, whose critical points are both strictly pre-periodic, a case not previously studied. For this map, the critical points are $1$ and $\infty$, and the branch points are $\infty$ and $0$. The critical points of $f$ have the following dynamics:
\begin{align*}
\begin{gathered}
\xymatrix{
	1 \ar[r]^2 & 
	\infty \ar[r]^2 &
	0 \ar[r] & 
	2 \ar@(ur,dr)[]
}
\end{gathered}
\end{align*}
where the label $2$ on an arrow indicates ramification. The map $f$ is a Latt\`es map. It is conjugate to the map $g(x)=-\frac{(x+1)^2}{4x}$,
which belongs to the family of Latt\`es maps arising from elliptic
curves with complex multiplication by $\mathbb{Z}[i]$ considered in
\cite[Example~4.7]{foz}.

We describe the structure of $\ga(f)$, bound the orders of its finite-level truncations, and characterize when the arboreal Galois groups $G_a(f)$ achieve their maximal size, i.e. when they are equal to the arithmetic iterated monodromy group. We denote the $n$th level of the arithmetic iterated monodromy group as $\ga_n$. 
For  $n\geq 1$, let $K_{n,a}$ be the splitting field of $f^n(x)-a$ over $k$ and $G_{n,a}$ be the Galois group of $K_{n,a}$ over $k$. Then $G_a(f)$ is isomorphic to the inverse limit of the groups $G_{n,a}(f)$. We now state one of our main results.

\begin{theorem}\label{intromain}
The arboreal Galois group $G_a(f)$ equals the arithmetic iterated monodromy group $\ga(f)$ if and only if they agree on the fourth level i.e. $G_{4,a}(f) = \ga_4(f)$.	
\end{theorem}

We also show that the arithmetic iterated monodromy group exhibits strong rigidity properties: its Hausdorff dimension is zero (\cref{Hdimension}), and it contains no odometers (\cref{Gaodometer}), i.e. no elements acting transitively on each level of the tree. In particular, we have the following result bounding the order of the arithmetic iterated monodromy group (\cref{orderGa}). See Section~\ref{ga} for details.

\begin{theorem}\label{ga bound}
 For any $n\geq 4$, the order of $\ga_n$ is at most $2^{2n}$. Moreover, for $k=\QQ$ and  $n \in \{4,5\}$, this upper bound is achieved. 
 \end{theorem}

While \cref{orderGa} gives the bound for all $n\geq 4$, its proof is based on calculating the order of $\ga_4$ over $k=\QQ$. Since we know that the Galois groups can only decrease under specialization, it suffices to exhibit a single value of
$a$ for which the maximum is achieved. We compute on Magma that for $a=5$ and $n=5$, the group $G_{a,n}(f)$ has order $2^{10}$. We note that over a general number field
$k$, the order may be strictly smaller.

Let $K_{\infty}$ denote the union of $K_{n}$, where $K_n$ is the splitting field of $f^n(x)-t$ over $k(t)$ and $t$ is  transcendental over $k$. The  Galois groups $\ga_n=\Gal(K_n/k(t))$ form an inverse system whose inverse limit gives the arithmetic iterated monodromy group (IMG). Replacing $k(t)$  by $\bar{k}(t)$ gives the geometric iterated monodromy group $\ggeo(f)$ and its finite truncations $\ggeo_n(f)$ for all $n \geq 1$. When $f$ is PCF, $\ggeo(f)$ is a finitely generated profinite group.  

Our analysis starts on the geometric side, and what makes it possible to push it
as far as we do is that $\ggeo(f)$ is unusually small: it is close to being
abelian.

\begin{theorem}\label{introabelian}
The group $\ggeo(f)$ contains a closed abelian subgroup $U$ of index $4$ where $U$ is a free $\ZZ_2$-module of rank two.
\end{theorem}

The subgroup $U$ carries more structure than this statement records. Conjugation
by the generator $a_3$ turns $U$ into a free module of rank one over the ring
$\ZZ_2[i]$ of integers of $\QQ_2(i)$, with $i$ acting as conjugation by $a_3$
(\cref{U torsion free}). This is where the arithmetic of the ramified quadratic
extension $\QQ_2(i)/\QQ_2$ enters, and it governs everything that follows: the
normalizer of $\ggeo(f)$ is built out of the unit group $\ZZ_2[i]^{\times}$ and
the Galois group of $\QQ_2(i)/\QQ_2$ (\cref{intronormalizer}), and the constant
field of $K_{\infty}$ is controlled by the norm map of that same extension
(\cref{introconstantfield}).

\begin{remark}
In view of the Latt\`es structure of $f$, the $\mathbb{Z}_2[i]$-module
structure of $U$ (Proposition~\ref{U torsion free}), as well as the descriptions
of the normalizer and of the constant field obtained in
Section~\ref{ga}, reflects the complex multiplication of the
underlying elliptic curve. Our proofs, however,  do not make use of this structure.
Instead, they work directly with the action on the tree, yielding
explicit recursive generators for the arithmetic IMG of $f$.
\end{remark}

\begin{remark}
Since $U$ is of finite index in $\ggeo(f)$,  the groups $U$ and $\ggeo$ have the same Hausdorff dimension;
together with \cite[Theorem~1]{Leung-Petsche} this gives an alternative proof of $\ggeo(f)$ having zero Hausdorff dimension.
By \cite[Theorem~1]{Leung-Petsche}, we have that $\ga(f)$ also has zero Hausdorff dimension.
We thank the authors for bringing their work to our attention. In this article, we show (\cref{Hdimension}) that the Hausdorff dimension of both groups is zero using \cref{ga bound}.
\end{remark}

\begin{theorem}\label{ggeo order}
 For any $n\geq 3$, the order of $\ggeo_n(f)$ is equal to $2^{n+2}$. 
\end{theorem}

We completely determine the normalizer of $G$ and the arithmetic IMG. Neither description depends on the specific form of the
map, only on the postcritical portrait of $f$.

\begin{theorem}\label{intronormalizer}
Let $N$ denote the normalizer of the geometric IMG of $f$. Then
\[ N/\ggeo \simeq (\ZZ_2[i]^{\times}/\langle i \rangle) \rtimes\langle\tau\rangle, \]
where $\tau$ is the complex conjugation. In particular,
$N/\ggeo$ is non-abelian.
\end{theorem}

Since $\ga(f)$ normalizes $\ggeo(f)$, we always have
$\ggeo(f)\subseteq\ga(f)\subseteq N$, so \cref{intronormalizer} measures how far
the arithmetic group can possibly move away from the geometric one. When
$\zeta_8\notin k$ — in particular for $k=\QQ$ — this upper bound is attained.

\begin{theorem}\label{introarith}
Assume that $k$ does not contain $\zeta_8$. Then, the normalizer of the geometric IMG of $f$ equals the arithmetic IMG of $f$. Moreover, the maximal abelian quotient of $\ga/G$ is isomorphic to $\ZZ_2^{\times}$.
\end{theorem}

Let $F$ denote the intersection of $K_{\infty}$ and $\bar{k}$. We call the field $F$ the constant field of $K_{\infty}$ and have the description $\Gal(F/k)\simeq\ga(f)/\ggeo(f)$. 
By \cref{intronormalizer,introarith}, the Galois group  $\Gal(F/k)$ 
is non-abelian, so $F/k$ is a non-abelian extension. By \cref{introarith}, its
maximal abelian subextension has Galois group $\ZZ_2^{\times}$, which is exactly
the Galois group of the $2$-power cyclotomic tower. Over $\QQ$, these two facts
can be combined with the ramification of $K_{\infty}$ to identify that
subextension completely.

\begin{theorem}\label{introconstantfield}
Let $k=\QQ$ and let $F=K_{\infty}\cap\overline{\QQ}$ denote the constant field.
Then,
\[ F\cap\QQ^{\mathrm{ab}}=\QQ(\mu_{2^{\infty}}) . \]
In particular, $K_{\infty}$ contains all $2$-power roots of unity, and the
abelianization of $\Gal(F/\QQ)\simeq\ga/G$ is the $2$-adic cyclotomic character.
\end{theorem}

So far, for all the quadratic PCF maps whose associated Galois groups have been studied, this extension turns out to be abelian; in fact, it is generated by some root of unity. By \cref{intronormalizer} and \cref{introarith}, the map $f(x)=\frac{2}{(x-1)^2}$ is the first example for
which it is not, and \cref{introconstantfield} shows that what survives of the
classical picture is exactly its abelian part.

\subsection{Outline}
We first describe $\ggeo(f)$, and then use the fact that $\ga(f)$ normalizes $\ggeo(f)$ to describe $\ga(f)$. \Cref{background} collects the background on the regular rooted binary tree, on its automorphism group, and on iterated monodromy groups. In \Cref{sec: group theory for exa 7 of LMY paper} we give recursive descriptions of the topological
generators of $\ggeo(f)$, construct an explicit model $G$ for the geometric iterated monodromy group, prove the semirigidity property that justifies this
model, and identify the abelian subgroup $U$ of index four; the orders of the finite levels computed there give \cref{ggeo order}. \Cref{sec:discriminant} is
devoted to the discriminants of the iterates of $f$, which are needed both for \cref{intromain} and for the ramification of the constant field. \Cref{sec:proof} determines the field generated at the fourth level and the Frattini subgroup of $\ga(f)$, and includes a proof of \cref{intromain}. In \Cref{Norm of G}
we study the normalizer $N$ of $G$ in $\Aut(T)$: we show that $U$ is a free $\ZZ_2[i]$-module of rank one, completing \cref{introabelian}; that $N$ contains
no odometers; and that we bound the orders $|N_n|$, yielding \cref{ga bound} and the vanishing of the Hausdorff dimension. Finally, \Cref{ga} determines the
quotients $N/U$ and $N/\ggeo$ and proves \cref{intronormalizer}, \cref{introarith} and \cref{introconstantfield}.


\subsection{Acknowledgments}
We thank Robert L. Benedetto for useful discussions and helpful comments. This project began at the Women in Numbers Europe (WINE) 5 conference in Split. 
We would like to thank the organizers and the sponsors of WINE 5.
The first author was supported by the Scientific and Technological Research Council of Türkiye (TÜBİTAK) under Grant 124F203.
The second author was supported by the research grant SONATA 20 "Symmetries of curves in positive characteristic" UMO-2024/55/D/ST1/01377 awarded by National Science Centre, Poland.
The fourth author was funded by the Deutsche Forschungsgemeinschaft (DFG, German Research Foundation) TRR 326 \textit{Geometry and Arithmetic of Uniformized Structures}, project number 444845124.
The fifth author was supported by Academia Sinica Investigator Grant AS-IA-112-M01 and NSTC grant 113-2115-M-001-001.

\section{Background}\label{background}
In this section, we present the necessary background on automorphisms of regular rooted binary trees and iterated monodromy groups. One may refer to \cite{EKO} for a more detailed overview.

\subsection{The regular rooted binary tree}\label{sec:tree}

Let $T$ be the infinite regular rooted binary tree whose vertices are finite words over the alphabet $\{1,2\}$. The root of the tree is denoted by the empty word $()$. For any integer $n \geq 1$, let $T_n$ denote the finite rooted subtree consisting of vertices of length at most $n$. The set of words of length $n$ is called level $n$ of $T$, and its elements are called the vertices at that level. The vertices of $T_n$ at level $n$ are called leaves.

\subsection{Automorphism group of $T$}

A bijection between the vertices of two regular rooted binary trees $T$ and $T'$ that preserves the tree structure is called an isomorphism. We denote the group of automorphisms of $T$ by $\Omega := \Aut(T)$ and write $\Omega_n := \Aut(T_n)$ for the group of automorphisms of the finite tree $T_n$.
We embed $\Omega \times \Omega$ into $\Omega$ by identifying the complete subtrees rooted at level $1$ of $T$ with $T$ itself. The image of this embedding $\Omega \times \Omega \into \Omega$ consists of the automorphisms acting trivially on the first level.

These identifications give rise to a splitting of the exact sequence
\[
1 \to \Omega \times \Omega \to \Omega \to S_2 \to 1
\]
yields the semidirect product decompositions
\[
\Omega \simeq (\Omega \times \Omega) \rtimes S_2 \quad \text{and} \quad
\Omega_n \simeq (\Omega_{n-1} \times \Omega_{n-1}) \rtimes S_2.
\]
Equivalently, $\Omega$ and $\Omega_n$ have a wreath product structure:
\begin{equation}\label{eq:wr}
\Omega \simeq \Omega \wr S_2 \quad \text{and} \quad
\Omega_n \simeq \Omega_{n-1} \wr S_2
\end{equation}
for $n \geq 2$.

This isomorphism in \eqref{eq:wr} is induced by the two complete subtrees of $T_n$ at level $1$, each of which is a copy of $T_{n-1}$. Hence, we may write elements of $\Omega_n$ as $(u,v)\tau$ with $u,v \in \Omega_{n-1}$ and $\tau \in S_2$.

We have the following relation in $\Omega$ arising from the wreath product:
\begin{equation}\label{eq:relations}
(x_1, x_2)\tau (y_1, y_2)\tau' = (x_1 y_{\tau(1)}, x_2 y_{\tau(2)}) \tau \tau'.
\end{equation}

Let $\sigma = (12) \in S_2$. We also use $\sigma$ to denote the automorphism that permutes the two subtrees at level $1$ of $T_n$. With the above notation,
\begin{equation}\label{eq:sigma}
\sigma := (\id, \id)\sigma \in \Omega_n \quad \text{for } n \geq 2,
\end{equation}
where $\id$ denotes the identity automorphism in $\Omega_{n-1}$.

\subsection{Automorphisms on finite levels}

For every $n \geq 1$, we write $\pi_n$ for the natural projection
\begin{equation}\label{not:pi}
\pi_n : \Omega \to \Omega_n,
\end{equation}
which corresponds to restricting the action of an element of $\Omega$ to the subtree $T_n$ consisting of levels $0,1,\ldots,n$.

Similarly, for any $m \geq n$, we denote the natural projection $\Omega_m \to \Omega_n$ by $\pi_{m,n}$. We abuse notation and write $\pi_n$ whenever the domain is clear. The image of an element $w \in \Omega$ (or $\Omega_m$ for $m \geq n$) under $\pi_n$ is denoted by $w\restr{T_n}$.

Let $H$ be a subgroup of $\Omega$. For each $n \geq 1$, define $H_n := \pi_n(H) \subset \Omega_n$.

\subsection{Closed subgroups of $\Omega$}

Since $\Omega$ is the inverse limit of the system $(\Omega_n, \pi_n)_{n \geq 1}$, it carries the profinite topology. For $c_1, \ldots, c_k \in \Omega$, let $\llang c_1, \ldots, c_k \rrang$ denote the topological closure of the subgroup $\langle c_1, \ldots, c_k \rangle$ in $\Omega$. We say that a subgroup $H$ is topologically generated by $c_1, \ldots, c_k$ if
\[
H = \llang c_1, \ldots, c_k \rrang.
\]

If $H$ is a closed subgroup of $\Omega$, then $H$ is the inverse limit of the system $(H_n)_{n \geq 1}$. In particular, every closed subgroup of $\Omega$ is a pro-$2$ group.

\subsection{Signs}
For any $n \geq 1$, the group $\Omega_n$ acts faithfully on the $n$th level of the tree $T$, which allows us to embed $\Omega_n$ into the symmetric group $S_{2^n}$. We let $\sgn_n$ denote the sign of the induced permutation on level $n$, which defines a continuous homomorphism
\begin{equation}
\sgn_n\colon \Omega \to \{\pm 1\}.
\end{equation}

For $\alpha=(\alpha_0,\alpha_1)\tau \in \Omega$ and $n>1$,  we have $$\sgn_n(\alpha)=\sgn_{n-1}(\alpha_0)\sgn_{n-1}(\alpha_1)$$  since $\sgn_n(\sigma)=1$ for all $n>1$.

\subsection{Conjugation in $\Omega$}
The following lemma tells us when two automorphisms are conjugate in $\Omega$. See \cite[Lemma~4.3]{Ejder26} or \cite[Lemma 1.3.1]{Pinkpolyn}.

\begin{lemma}\label{conjugation rules in Aut T} Suppose $u, u', v, v' \in \Omega$.
\begin{enumerate}[label=(\roman*)]
    \item Let $\tau, \tau' \in \langle \sigma \rangle$. If $(u,v)\tau \sim (u',v')\tau'$, then $\tau = \tau'$.
    \item We can verify whether two automorphism which are trivial on the first level are conjugate via the criterion
    \[
    (u,v) \sim (u',v') \Longleftrightarrow
    (u \sim u' \text{ and } v \sim v') \text{ or }
    (u \sim v' \text{ and } v \sim u').
    \]
    \item For automorphisms that are non-trivial on the first level, we have the conjugation rule
    \[
    (u,v)\sigma \sim (u',v')\sigma \Longleftrightarrow uv \sim u'v'.
    \]
\end{enumerate}
\end{lemma}

\subsection{Odometers}
An automorphism $\alpha \in \Omega$ is called an \emph{odometer} if $\alpha$ acts as transitively on every level $n$ of the tree. Equivalently, $\alpha$ is an odometer if $\alpha \restr{T_n}$ has order $2^n$ for all $n\geq 1$. By definition, any conjugate of $\alpha$ in $\Omega$ is also an odometer. The existence of odometers in a subgroup of $\Omega$ can be tested by the following proposition. See \cite[Proposition~1.6.2]{Pinkpolyn} or \cite[Proposition 2.1]{Ejder26} for a proof.
\begin{proposition}\label{pink:odometer}
An element $\gamma \in \Omega$ is an odometer if and only if $\sgn_n(\gamma)=-1$ for all $n\geq 1$.
\end{proposition}


\subsection{Iterated monodromy groups}\label{sec:mon}
Let $k$ be a number field. Fix an algebraic closure $\bar{k}$ of $k$ and denote the absolute Galois group as $G_k = \Gal(\bar{k}/k)$.
Let $f:\PP^1_k \to \PP^1_k$ be a morphism of degree $d \geq 2$ defined over $k$. Let $C$ be the set of critical points of $f$ and let $P$
be the forward orbit of the points in $C$, i.e.,
\[ P:=\{ f^n(c) : n\geq 1, c \in C\}. \]

We set $X= \PP^1_k \backslash P$ and let $x_0 \in X(k) $. Since each $f^n$ is a connected unramified covering of $X$, it is determined  by the monodromy action of $\pi_1^{\acute{e}t}(X, x_0)$ on $f^{-n}(x_0)$ up to isomorphism. Let $T_{x_0}$ be the tree defined as follows: it is rooted at $x_0$, and the vertices of $T_{x_0}$ at level $n$ are the points of $f^{-n}(x_0)$ for all $n \geq 1$, and two vertices $p,q$ are connected if $f(p)=q$. As $n$ varies, associated monodromy defines a representation 
\begin{equation}\label{eq:rho}
	\rho \colon \pi_1^{\acute{e}t}(X, x_0) \to \Aut(T_{x_0}).
\end{equation} 
We call the image of the map the \emph{arithmetic iterated monodromy group} $G^{\text{arith}}(f)$ of $f$. One can also study this representation over $\bar{k}$ and obtain 
\[ \pi_1^{\acute{e}t}(X_{\bar{k}}, x_0) \to \Aut(T_{x_0}),
\]
whose image in this case is called the \emph{geometric iterated monodromy group} $G^{\text{geom}}(f)$.
We note that these iterated monodromy groups are unique up to conjugation by the elements of $\Aut(T_{x_0})$. 
Moreover, we point out that we can view $\ggeo(f)$ as a subgroup of $\ga(f)$.
The arithmetic and the geometric iterated monodromy groups of $f$ fit into an exact sequence as follows:
\begin{equation}\label{eq:exact}
	\begin{tikzcd}
		1\arrow{r}  &  \pi_1^{\acute{e}t}(\PP_{\bar{k}}^1 \backslash P, x_0) \arrow{r} \arrow{d}& \pi_1^{\acute{e}t}(\PP_k^1 \backslash P, x_0) \arrow{r} \arrow{d} &G_k \arrow{r} \arrow{d} &1  \\
		1 \arrow{r} &   G^{\text{geom}}(f) \arrow{r} &G^{\text{arith}}(f) \arrow{r}  &\Gal(F /k) \arrow{r} &1 \\
	\end{tikzcd}
\end{equation}
for some field extension $F$ of $k$. This field $F$ is called the \emph{constant field} subextension and determining this field $F$ or its degree is a fundamental problem.

Another way to describe the geometric (arithmetic resp.) iterated monodromy group of $f$ is  as the projective limit of the Galois group of the splitting field of the $n$th iteration $f^{n}(x)-t$ over $\overline{k}(t)$ (over $k(t)$ resp.). By abuse of notation, we denote the geometric and the arithmetic iterated monodromy groups of $f$ as follows:
\[ G^{\text{geom}}(f)=\ilim_n \Gal(\overline{k}(f^{n}(x)-t)/\overline{k}(t)) \text{ and }  G^{\text{arith}}(f)=\ilim_n \Gal(k(f^{n}(x)-t)/k(t)). \]

The groups $G^{\text{geom}} (f)$ and $G^{\text{arith}}(f)$ are profinite groups and they are embedded into $\Aut(T_{x_0})$ by construction. We identify the tree $T_{x_0}$ we constructed above with the infinite binary tree $T$ we described in the previous section. From now on, we  assume $G^{\text{geom}}(f)$ and $G^{\text{arith}}(f)$ are subgroups of $\Omega=\Aut(T)$. One can find a labeling of these trees which realize $G^{\text{geom}}(f)$ and $G^{\text{arith}}(f)$ as  self-simlar groups.

Whenever it is clear from the context, we omit $f$ from the notation and write $G^{\text{geom}}$ and $G^{\text{arith}}$. \\

The rational function $f$ is called \emph{postcritically finite (PCF)} if its postcritical set $P$ is finite. For a PCF function $f$,  the geometric iterated monodromy group is topologically finitely generated, with generators indexed by the elements of $P$.

\subsection{Arboreal Galois Groups}
Let $k$ be a number field, and $\bar{k}$ be a separable closure of $k$. 
We write the absolute Galois group as $G_k= \Gal(\bar{k}/k).$
Suppose $f \colon \PP^1_k \to \PP^1_k$ is a morphism of degree $d \geq 2$ defined over $k$, and
fix $a \in k$. Consider the morphism $s_a \colon \Spec(k) \to \PP^1_k$
corresponding to $a$, and denote its image by $\bar{a}$. 
From now on, we assume $ \bar{a} \notin P$ such that we can build a regular rooted $d$-ary tree.
By functoriality of the \'etale fundamental group, we obtain a homomorphism
\[
G_k \to \pi_1^{\acute{e}t}(\PP^1_k \setminus P, a).
\]
Composing with $\rho \colon \pi_1^{\acute{e}t}(\PP^1_k \setminus P, a) \to \Aut(T_{a})$,  yields a representation
\[
\rho_{f,a} \colon G_k \to \Aut(T_{a})
\]
which agrees with the natural action of $G_k$ on the tree $T_{a}$ consisting of the preimages of $a$ under iterates of $f$.
The image of $\rho_{f,a}$ is called the \emph{arboreal Galois group}
associated to $(f,a)$, and we denote it by $G_{a}(f) = G_a$.
Since a change of base point corresponds to conjugation by an element of $\Omega$, the group
$G_{a}(f)$ is conjugate to a subgroup of the arithmetic iterated monodromy
group of $f$.

As the name indicates one can also understand the arboreal Galois group as a Galois group. 
We write the splitting field of $f^n(x)-a$ over $k$ as $K_{n,a}$. 
We set $K_{\infty,a}$ as the union of all $K_{n,a}$ and define  $G_{n,a} = \Gal(K_{n,a}/k).$ 
Then the arboreal Galois group $G_{a}(f)$ is isomorphic to the Galois group of the field extension $K_{\infty,a}/k$ and likewise to the inverse limit of the $G_{n,a}$. 

\section{Group theory for $f(x)= 2/(x-1)^2$} \label{sec: group theory for exa 7 of LMY paper}
From now on, we consider the rational function given by
\[
f(x) = \frac{2}{(x-1)^2}.
\]
The critical points of $f$ have the following dynamic
\begin{align*}
\begin{gathered}
\xymatrix{
1 \ar[r]^2 & 
\infty \ar[r]^2 &
0 \ar[r] & 
2 \ar@(ur,dr)[]
}
\end{gathered}
\end{align*}
where the label $2$ indicates ramification. 
Hence, the rational map $f$ is PCF.
For this map, the critical points are $1$ and $\infty$, and the branch points are $\infty$ and $0$. We set
\[
C = \{1,\infty\}, \qquad P = \bigcup_{n} f^n(C) = \{0,2,\infty\}.
\]

The group $G^{\text{geom}}$ is topologically generated by elements $b_p$ for all points $p$ in the postcritical set $P$.
Since $P$ is finite, the group $G^{\text{geom}}= \ggeo(f)$ is topologically finitely generated, and the product of the elements $b_p$ (taken in some order) is equal to the identity. These generators $b_p$ are conjugate, in $\Omega$, to certain elements described in \cite[Proposition 1.7.15]{Pinkpolyn}.

Using Pink's work (see \cite[Proposition 1.7.15]{Pinkpolyn}), we may write
\[
\ggeo = \llang b_\infty, b_0, b_2 \mid b_\infty b_0 b_2 = \id \rrang,
\]
with
\begin{align*}
b_\infty &\sim \sigma, \\
b_0 &\sim (b_\infty, \id)\sigma, \\
b_2 &\sim (b_0, b_2).
\end{align*}

The triple $(b_\infty, b_0, b_2)$ lies in the (non-empty) set
\[
M \coloneqq \left\{ (c_1, c_2, c_3) \in \Omega^3 \;\middle|\;
c_1 \sim \sigma,\;
c_2 \sim (c_1, \id)\sigma,\;
c_3 \sim (c_2, c_3),\;
c_1 c_2 c_3 = \id
\right\}.
\]

\begin{lemma}
If we fix $(a_1, a_2, a_3) \in M$, then
\[
M = \left\{ (c_1, c_2, c_3) \in \Omega^3 \;\middle|\;
c_i \sim a_i \text{ for } i=1,2,3,\ \text{and } c_1 c_2 c_3 = \id
\right\}.
\]
Moreover, this set is closed under conjugation by elements of $\Omega$.
\end{lemma}

\begin{proof}
Let $(c_1, c_2, c_3) \in M$. Then $c_1 \sim \sigma \sim a_1$.
Using \cref{conjugation rules in Aut T}~(iii) and $c_1 \sim a_1$, we obtain
\[
c_2 \sim (c_1, \id)\sigma \sim (a_1, \id)\sigma \sim a_2.
\]

We next show that $c_3 \sim a_3$ by induction on the finite levels. For $n=1$, they are both identity. Assume that $c_3 \restr{T_{n-1}}$ is conjugate to $a_3 \restr{T_{n-1}}$ in $\Omega_{n-1}$. Then $c_3\restr{T_n} \sim (c_2\restr{T_{n-1}},c_3\restr{T_{n-1}}) \sim (a_2\restr{T_{n-1}}, a_3\restr{T_{n-1}})$ by the induction hypothesis.
Taking the limit as $n$ tends to infinity, we find that $c_3 \sim a_3$.

Conversely, let $c_1, c_2, c_3 \in \Omega$ satisfying $c_1 c_2 c_3 = \id$ and $c_i \sim a_i$ for $i=1,2,3$. Then
\[
c_1 \sim a_1 \sim \sigma,
\]
and
\[
c_2 \sim a_2 \sim (a_1, \id)\sigma \sim (c_1, \id)\sigma.
\]
As before, 
\[c_3 \sim a_3 \sim (a_2,a_3) \sim (c_2,c_3). \]
Hence $(c_1, c_2, c _3) \in M$.
The description of $M$ immediately shows that it is closed under conjugation by elements of $\Omega$.
\end{proof}

To work with explicit generators of $G^{\mathrm{geom}}$, we choose a particular triple in $M$ and show
that the chosen triple generates a group $G$ which is conjugate (in $\Omega$) to $G^{\mathrm{geom}}$, meaning
$G^\omega = G^{\mathrm{geom}}$ for some $\omega \in \Omega$.

This triple is defined by
\begin{align*}
a_1 &\coloneqq \sigma, \\
a_2 &\coloneqq (a_3^{-1}, a_2^{-1}) \sigma, \\
a_3 &\coloneqq (a_2, a_3).
\end{align*}
The idea is to fix $a_1$ and $a_3$, and then define $a_2 = a_1^{-1} a_3^{-1}$ so that $a_1 a_2 a_3 = \id$.

Applying \cref{conjugation rules in Aut T}, we obtain $a_2 \sim (a_1, \id)\sigma$, since
\[
a_3^{-1} a_2^{-1} = a_1^{-1} = a_1 \sim a_1.
\]
We note that $a_1$ has order $2$, while $a_2$ and $a_3$ have order $4$ in $\Omega$.

\medskip

Our first goal is the following theorem, which ensures that we may work with the group generated by the $a_i$ to study $\ggeo$. 
\begin{theorem}\label{ggeom is conjugate to group generated by a_i}
Suppose $b_1, b_2, b_3 \in \Omega$ satisfy $a_i \sim b_i$ for all $i=1,2,3$, and $b_1 b_2 b_3 = \id$. Then there exist $\beta \in \Omega$ and $g_1, g_2, g_3 \in \llang a_1, a_2, a_3 \rrang$ such that
\[
b_i = \beta\, g_i a_i g_i^{-1} \beta^{-1}
\quad \text{for all } i=1,2,3.
\]
In particular, the topological closures of the groups generated by $a_1,a_2,a_3$ and $b_1,b_2,b_3$ are conjugate.
\end{theorem}

The property given in Theorem~\ref{ggeom is conjugate to group generated by a_i} is called the semirigidity property by Pink. See \cite[Theorem~2.4.1 and Theorem~3.4.1]{Pinkpolyn} for a similar result.

\subsection{Constructing a Model for the geometric iterated monodromy group}
We define the group
\[
G \coloneqq \llang a_1, a_2, a_3 \rrang = \llang a_1, a_3 \rrang,
\]
to be the closed subgroup of $\Omega$ generated by $a_1, a_2$, and $a_3$. We first observe that $\sgn_4(a_1)=1$ and $\sgn_4(a_3)=\sgn_3(a_3)^2$, hence equals $1$ as well. Since $G$ is generated by $a_1$ and $a_3$, $\sgn_4(\alpha)=1$ for all $\alpha \in G$, hence $G_4$ contains no elements of order $2^4$. By ~\cref{pink:odometer}, $G$ contains no odometers.

For each $i = 1,2,3$, we set
\[
H_i \coloneqq \overline{\llang a_i \rrang},
\]
which is the normal closure of the closed subgroup generated by $a_i$. We denote the projection of $H_i$ onto level $n$ by $H_{i,n}$ for all $i=1,2,3$. 
From the definition, the orders of the $a_i$ and a computation on level $4$, we compute the indices
\[
[G:H_1]= 4,
\quad [G:H_2] = 2,
\quad [G:H_3] = 2.
\]
Since $G$ is generated by $a_1,a_2,a_3$, we have $G=H_1H_2=H_1H_3=H_2H_3$.

We note that $G_n/H_{i,n} \to G_{n-1}/H_{i,n-1}$ is surjective for all $i \in \{1,2,3\}$. By the recursive definitions of the generators, if $(u,v)\tau \in G$ for some $\tau \in S_2$, then $u,v$ are both in $G$ (self-similarity). Also for any $g\in G$, there is some $g' \in G$ such that $(g, g')\tau \in G$ (fractal property). Since $\sigma \in G$, one may take $\tau=\id$.

The description $G \cap (\Omega \times \Omega) = H_3$ follows from the facts
\begin{enumerate}[label=(\alph*)]
    \item $H_3 \subseteq G \cap (\Omega \times \Omega)$,
    \item $[G:H_3] = 2$,
    \item $[G: G \cap (\Omega \times \Omega)] = \# \langle \sigma \rangle = 2$. 
\end{enumerate}
We denote by $\llbracket G, G \rrbracket$ the topological closure of the commutator subgroup of $G$.

\begin{lemma}\label{commutator of G}
The closure of the commutator subgroup of $G$ satisfies
\[
\llbracket G, G \rrbracket = H_1 \cap H_3.
\]
Moreover, for any $n \geq 3$, the commutator subgroup $\llbracket G_n, G_n \rrbracket$ is equal to $H_{1,n} \cap H_{3,n}$ and has index $8$ in $G_n$. In particular, the maximal abelian quotient of $G$ is $G/H_1 \times G/H_3$ which is isomorphic to $\ZZ/2\ZZ \times \ZZ/4\ZZ$. 
\end{lemma}

\begin{proof}
Consider the group homomorphism
\[
G \to G/H_1 \times G/H_3.
\]
Since $G = H_1 H_3$, the Chinese remainder theorem for groups implies that this homomorphism is surjective with kernel $H_1 \cap H_3$. Hence $H_1 \cap H_3$ has index $8$ in $G$.

Since $G$ is generated by $a_1$ and $a_3$, and $a_1$ has order $2$ while $a_3$ has order $4$, the order of the quotient $G/\llbracket G, G \rrbracket$ is at most $8$. As $G/H_1 \times G/H_3$ is abelian, it follows that $\llbracket G, G \rrbracket \subseteq H_1 \cap H_3$. Therefore, we obtain
\[
\llbracket G, G \rrbracket = H_1 \cap H_3.
\]

For the finite groups $G_n$, the element $a_1\restr{T_n}$ has order $2$ and $a_3\restr{T_n}$ has order $4$ for all $n \geq 3$. The rest of the argument is analogous.
\end{proof}

We next compute an explicit set of generators for the commutator $\llbracket G, G \rrbracket$.

\begin{lemma}\label{betas}
The group $H_1 \cap H_3$ is generated by the elements $a_1 a_3 a_1 a_3^{-1}$ and $a_3^{-1} a_1 a_3 a_1$.
\end{lemma}

\begin{proof}
Let $\beta_1 := a_1 a_3 a_1 a_3^{-1}$ and $\beta_2 := a_3^{-1} a_1 a_3 a_1$. We have
\[
\beta_1 = (a_3 a_2^{-1}, a_2 a_3^{-1}), 
\quad 
\beta_2 = (a_2^{-1} a_3, a_3^{-1} a_2).
\]
Since $G$ is generated by $a_1$ and $a_3$, the following computations show that the subgroup $\llang \beta_1, \beta_2 \rrang$ is normal in $G$.

\begin{align*}
a_1 \beta_1 a_1^{-1} &= a_3 a_1 a_3^{-1} a_1^{-1} = \beta_1^{-1}, \\
a_1 \beta_2 a_1^{-1} &= a_3 a_1 a_3^{-1} a_1^{-1} = \beta_2^{-1}, \\
a_3 \beta_1 a_3^{-1}
&= (a_2, a_3)(a_3 a_2^{-1}, a_2 a_3^{-1})(a_2^{-1}, a_3^{-1}) \\
&= (a_2 a_3 a_2^{-2}, a_3 a_2 a_3^{-2}) \\
&= (a_1 a_2^2, a_3 a_1 a_3) \\
&= \beta_2^{-1}, \\[4pt] 
a_3 \beta_2 a_3^{-1}
&= (a_2, a_3)(a_2^{-1} a_3, a_3^{-1} a_2)(a_2^{-1}, a_3^{-1}) \\
&= (a_3 a_2^{-1}, a_2 a_3^{-1}) \\
&= \beta_1.
\end{align*}

Moreover, the quotient $G / \llang \beta_1, \beta_2 \rrang$ is abelian, which implies that $\llbracket G, G \rrbracket \subseteq \llang \beta_1, \beta_2 \rrang$. \\
Since $\beta_1, \beta_2 \in \llbracket G, G \rrbracket$, the claim follows.
\end{proof}

We define $U$ to be the normal closure of the topological subgroup generated by $a_2 a_3^{-1}$, i.e.
\[
U \coloneqq \overline{\llang a_2 a_3^{-1} \rrang}.
\]
By definition $U$ is a normal subgroup of $G$ and we have $G=U\cdot \langle a_3\rangle$.

\begin{lemma}\label{index of U}
The index of $U$ in $G$ is equal to $4$. Moreover, for any $n \geq 3$, we have $[G_n : U_n] = 4$.
\end{lemma}

\begin{proof}
Since $G$ is topologically generated by $a_2$ and $a_3$, the quotient $G/U$ is generated by the image of $a_2$ (equivalently, the image of $a_3$). Since $a_3$ has order $4$, it follows that $|G/U|$ divides $4$. A computation for $n=3$ shows that $|G_3 / U_3| = 4$, hence $|G/U| \geq 4$, and therefore $|G/U| = 4$.
\end{proof}

\begin{proposition}\label{lem:U-semi}
The group $G$ is the semidirect product of $U$ and $\langle a_3 \rangle$.
\end{proposition}
\begin{proof}
Since $G$ is generated by $a_2, a_3$ and $U$ contains $a_2a_3^{-1}$,  $G=U \cdot \langle a_3 \rangle$. By \cref{index of U}, the index of $U$ in $G$ is equal to $4$. Since $a_3$ has order $4$, it must hold that $U \cap \langle a_3 \rangle=\{\id\}$.
\end{proof}

\begin{lemma}\label{lem:Ugen}
The subgroup $U$ is topologically generated by $\gamma_1 := a_2 a_3^{-1}$ and $\gamma_2 := a_3^{-1} a_2$.
\end{lemma}

\begin{proof}
We first calculate that
\begin{align*}
\gamma_1 &= a_2 a_3^{-1}
= (a_3^{-1}, a_2^{-1})\sigma (a_2^{-1}, a_3^{-1})
= (a_3^{-2}, a_2^{-2})\sigma, \\
\gamma_2 &= a_3^{-1} a_2
= (a_2^{-1} a_3^{-1}, a_3^{-1} a_2^{-1}) \sigma.
\end{align*}

Since $\gamma_2 = a_3^{-1} \gamma_1 a_3$, we have $\llang \gamma_1, \gamma_2 \rrang \subseteq U$. Conversely, $U \subseteq \overline{\llang \gamma_1, \gamma_2 \rrang}$, so it suffices to show that $\llang \gamma_1, \gamma_2 \rrang$ is normal in $G$.

It is enough to verify that
\[
g \llang \gamma_1, \gamma_2 \rrang g^{-1} \subseteq \llang \gamma_1, \gamma_2 \rrang
\]
for generators $g$ of $G$. Using $a_1 a_2 a_3 = \id$ and $a_1^2 = a_2^4 = a_3^4 = \id$, we compute:

For $g = a_1$,
\begin{align*}
a_1 \gamma_1 a_1^{-1}
&= a_1 a_2 a_3^{-1} a_1^{-1}
= a_3^{-2} a_1
= a_3 a_2^{-1}
= \gamma_1^{-1},
\end{align*}
and
\begin{align*}
a_1 \gamma_2 a_1^{-1}&=a_1 a_3^{-1} a_2 a_1^{-1} = (a_1a_2^2)^{-1} = (a_3^{-1}a_2)^{-1}= \gamma_2^{-1}.
\end{align*}

For $g = a_3$,

\begin{align*}
a_3 \gamma_1 a_3^{-1} &= a_3 a_2 a_3^{-1} a_3^{-1} = a_3 a_2 a_3^2 =  a_3 a_1 a_3 = a_2^{-1} a_3 = \gamma_2^{-1},
\end{align*}
and
\begin{align*}
a_3 \gamma_2 a_3^{-1} &= a_3 a_3^{-1} a_2 a_3^{-1} = \gamma_1.
\end{align*}

Thus $\llang \gamma_1, \gamma_2 \rrang$ is normal in $G$, and the claim follows.
\end{proof}

\begin{proposition}\label{Uabelian}
The subgroup $U$ is abelian.
\end{proposition}

\begin{proof}
By \cref{lem:Ugen}, the group $U$ is generated by $\gamma_1$ and $\gamma_2$, so it suffices to show that they commute.

We compute
\begin{equation}\label{gamma1gamma2}
\gamma_1 \gamma_2
= (a_3^{-3} a_2^{-1}, a_2^{-3} a_3^{-1})=(\gamma_1^{-1},\gamma_1),
\quad
\gamma_2 \gamma_1
= (a_2^{-1} a_3^{-1} a_2^{-2}, a_3^{-1} a_2^{-1} a_3^{-2}).
\end{equation}
Using the relations $a_1 a_2 a_3 = \id$, and the fact that $a_3$ has order $4$,
\begin{align*}
		a_3^{-3} a_2^{-1}(a_2^{-1} a_3^{-1} a_2^{-2})^{-1}&=a_3a_2^{-1}a_2^2a_3a_2=a_3a_1a_2 \\
        &=\id.
\end{align*}
	Analogously, using $a_1$ has order $2$, we find that
	\begin{align*}
a_2^{-3} a_3^{-1}(a_3^{-1} a_2^{-1} a_3^{-2})^{-1}&=a_2a_3^{-1}a_3^2a_2a_3=(a_2a_3)^2=a_1^2 \\
&=\id.
	\end{align*}
Hence $\gamma_1 \gamma_2 = \gamma_2 \gamma_1$.
\end{proof}

It is worth stating that the proof of \cref{lem:Ugen} shows that we have 
\begin{equation}\label{conjugation by a3}
 a_3 \gamma_1 a_3^{-1}=\gamma_2^{-1}, \quad a_3 \gamma_2 a_3^{-1}=\gamma_1
 \end{equation}
 and 
 \begin{equation}\label{conjugation by a1}
 a_1 \gamma_1 a_1^{-1}=\gamma_1^{-1} ,  \quad a_1 \gamma_2 a_1^{-1}=\gamma_2^{-1} \end{equation}
Moreover, a quick calculation shows that $a_3^2$ gives the automorphism of $U$ taking $\alpha \mapsto \alpha^{-1}$ for any $\alpha \in U$. Since $\gamma_1=a_2a_3^{-1}$ is in $U$ and $U$ is abelian by \cref{Uabelian},  the action of $a_2$ and $a_3$ on $U$ coincide.

\begin{equation}\label{conjugation by a2}
 a_2 \gamma_1 a_2^{-1}=\gamma_2^{-1}, \quad a_2 \gamma_2 a_2^{-1}=\gamma_1
 \end{equation}

\begin{corollary}\label{isomorphicU}
\hfill
\begin{enumerate}
    \item $\llbracket G, G \rrbracket = H_1 \cap H_3 = \{(x, x^{-1}) \mid x \in U\}$.
    \item For any $n \geq 2$, $H_{1,n} \cap H_{3,n}$ is isomorphic to $U_{n-1}$.
\end{enumerate}
\end{corollary}

\begin{proof}
Remember that $\gamma_1 = a_2 a_3^{-1}$ and $\gamma_2 = a_3^{-1} a_2$.
By \cref{betas}, $H_1 \cap H_3$ is generated by $\beta_1 = (\gamma_1^{-1}, \gamma_1)$ and $\beta_2 = (\gamma_2^{-1}, \gamma_2)$. Since $U$ is abelian and generated by $\gamma_1$ and $\gamma_2$, we obtain
\[
H_1 \cap H_3 = \{(x^{-1}, x) \mid x \in U\} = \{(x, x^{-1}) \mid x \in U\}.
\]

Define $\rho_2: G \cap (\Omega \times \Omega) \to G$ by $(x,y) \mapsto y$. The image of $H_1 \cap H_3$ is generated by $\gamma_1$ and $\gamma_2$, so by \cref{lem:Ugen}, this map induces an isomorphism $H_1 \cap H_3 \simeq U$.
\end{proof}

\begin{proposition}\label{orderG}
The order of $G_n$ is $2^{n+2}$ for all $n \geq 3$.
\end{proposition}

\begin{proof}
By \cref{index of U} and \cref{commutator of G}, for $n \geq 3$ we have $[G_n : U_n] = 4$ and $[G_{n+1} : H_{1,n+1} \cap H_{3,n+1}] = 8$. Using \cref{isomorphicU}, we obtain
\[
8 = [G_{n+1}: H_{1,n+1}\cap H_{3,n+1}]=\frac{|G_{n+1}|}{|G_n|} \cdot 4,
\]
hence $|G_{n+1}| / |G_n| = 2$.

A computation in MAGMA shows that $|G_3| = 2^5$, and therefore $|G_n| = 2^{n+2}$.
\end{proof}

\begin{corollary}\label{cor:uniqueG}
For every $x \in G$, there exists a unique $y \in G$ such that $(x,y) \in G$.
\end{corollary}

\begin{proof}
By \cref{orderG}, for $n \geq 4$ we have $|G_{n-1}| = |H_{3,n}|$. Since the projection
\[
\rho_1: H_{3,n} \to G_{n-1}
\]
is surjective, for every $x \in G_{n-1}$ there exists a unique $y \in G_{n-1}$ such that $(x,y) \in G_n$.

Passing to the limit gives the claim.
\end{proof}

\subsection{Conjugacy}

We now return to our goal of proving \cref{ggeom is conjugate to group generated by a_i}. We begin with several preparatory lemmas.

\begin{proposition}\label{characterization of U}
The subgroup $U$ of $\Omega$ is equal to
\[
\{ uv^{-1} \mid (u,v) \in G \}.
\]
\end{proposition}

\begin{proof}
Let $(x,y) \in G$. Then
\[
(x,y)a_1(x,y)^{-1}a_1^{-1} = (xy^{-1}, yx^{-1})
\]
is a commutator element in $G$, hence it lies in $H_1 \cap H_3$. By \cref{isomorphicU}, it follows that $xy^{-1} \in U$. Thus
\[
\{ uv^{-1} \mid (u,v) \in G \} \subseteq U.
\]

Using \cref{isomorphicU} and the fact that $H_3 = (H_1 \cap H_3)\cdot\llang a_3 \rrang$, every element $(x,y) \in G$ can be written in the form
\[
(x,y) = (\beta, \beta^{-1}) a_3^k
\]
for some $\beta \in U$ and $k \in \{0,1,2,3\}$. In particular,
\[
H_3 = \{ (\beta, \beta^{-1}) a_3^k \mid \beta \in U,\ k=0,1,2,3 \}.
\]

For $(x,y) \in H_3$, we obtain
\[
xy^{-1} = \beta^2 a_2^k a_3^{-k}.
\]

We compute
\[
a_2^k a_3^{-k} =
\begin{cases}
\id & \text{if } k=0, \\
\gamma_1 & \text{if } k=1, \\
a_2 \gamma_1 a_2^{-1} \gamma_1 & \text{if } k=2, \\
\gamma_2^{-1} & \text{if } k=3,
\end{cases}
\]
where $\gamma_1$ and $\gamma_2$ are as in \cref{lem:Ugen}. In each case, $a_2^k a_3^{-k} \in U$, and since $U$ is abelian, it commutes with $\beta$.

Consider the homomorphism $\phi: U \to U$, $x \mapsto x^2$. The quotient $U/U^2$ has four elements, represented by $\id$, $\gamma_1$, $\gamma_2$, and $\gamma_1\gamma_2$. As $\beta$ varies over $U$ and $k$ ranges over $\{0,1,2,3\}$, we obtain all cosets of $U^2$ in $U$: when $k=0,1,3$, we recover the elements in the coset containing $\id, \gamma_1$ and $\gamma_2$ respectively. When $k=2$, 
\begin{align*}  xy^{-1}=& \beta^2 (a_2\gamma_1 a_2^{-1} \gamma_1 ) \\
                                  = &\beta^2 \gamma_2^{-1}\gamma_1. 
\end{align*}
Hence $U \subseteq \{ uv^{-1} \mid (u,v) \in G \}$, completing the proof.
\end{proof}

\begin{proposition}\label{prop:centralizer}
Let $n\geq 3$. The centralizer of $a_1\restr{T_n}$ and $a_3\restr{T_n}$ in $G_n$ has order $8$.
\end{proposition}

\begin{proof}
Let $n\geq 3$. The conjugacy class of $a_1\restr{T_n}$ in $G_n$ is
\[
\{ (xy^{-1}, yx^{-1})\sigma \in G_n \mid (x,y) \in G_n \}.
\]
By~\cref{characterization of U}, its size is $|U_{n-1}|$, and hence the centralizer has order
\[
|C_{G_n}(a_1)| = \frac{|G_n|}{|U_{n-1}|}
= \frac{|G_n|}{|H_{1,n}\cap H_{3,n}|} = 8.
\]

For $a_3$, note that $C_{G_n}(a_3) \subseteq H_3$. The projection
\[
\rho_2: H_3 \to G, \quad (x,y) \mapsto y,
\]
induces an injective map $C_{G_{n+1}}(a_3) \to C_{G_n}(a_3)$ by \cref{cor:uniqueG}. A MAGMA computation shows that $|C_{G_3}(a_3)| = 8$, hence $|C_{G_n}(a_3\restr{T_n})| \le 8$ for all $n \ge 3$. Finally, for any $g \in G_n$, 
\[
ga_3 \restr{T_n}g^{-1}= (ga_3\restr{T_n}g^{-1}a_3\restr{T_n}^{-1})a_3\restr{T_n} \in \llbracket G_n,G_n \rrbracket a_3\restr{T_n},
\]
and hence the conjugacy classes lie in cosets of the commutator subgroup; their size is at most $|H_{1,n} \cap H_{3,n}|$. This yields equality, and hence $|C_{G_n}(a_3)| = 8$.
\end{proof}

\begin{corollary}
Let $n\geq 3$. The center $Z(G_n)$ of $G_n$ has order $2$.
\end{corollary}

\begin{proof}
Let $\alpha \in C_{G_n}(a_3\restr{T_n})$. Since $G = U \cdot \langle a_3 \rangle$, we can write $\alpha = ua_3^k\restr{T_n}$ for some $u \in U_n$ and $k \in \{0,1,2,3\}$. Hence $u \in C_{G_n}(a_3\restr{T_n})$ as well. As $u \in U_n \cap C_{G_n}(a_3\restr{T_n})$ and $U_n$ is abelian by \cref{Uabelian}, it follows that $u \in Z(G_n)$. Hence
\[
C_{G_n}(a_3\restr{T_n}) \subseteq Z(G_n)\cdot \langle a_3\restr{T_n} \rangle.
\]
It is straightforward to check that  $$Z(G_n)\cdot \langle a_3\restr{T_n} \rangle \subseteq C_{G_n}(a_3\restr{T_n}).$$

By ~\cref{prop:centralizer}, $ Z(G_n)\cdot \langle a_3\restr{T_n} \rangle$ has order $8$. Since $a_3^2\restr{T_n}$ does not commute with $a_1 \restr{T_n}$, we have that
$Z(G_n) \cap \langle a_3\restr{T_n} \rangle = \{\id\}$. Hence, $Z(G_n)$ must have order $2$.
\end{proof}

\begin{corollary}
The centralizer of $a_2\restr{T_n}$ in $G_n$ has order $8$.
\end{corollary}

\begin{proof}
Let $\alpha \in C_{G_n}(a_2\restr{T_n})$. We can find $k \in \{0,1,2,3\}$ and $u \in U_n$ such that $\alpha = ua_2^k\restr{T_n}$. Hence $u$ must lie in $C_{G_n}(a_2\restr{T_n})$. Since $U$ is abelian and $G=U \cdot \llang  a_2\rrang$, the element $u$ must be in $Z(G_n)$, which has only $2$ elements. Hence $C_{G_n}(a_2\restr{T_n})$ has exactly $8$ elements, as $Z(G_n) \cap \langle a_2 \restr{T_n}\rangle = \{\id\}$ and $|\langle a_2 \restr{T_n} \rangle|=4$.
\end{proof}

\begin{proposition}\label{inverse pair}
An automorphism $x$ lies in $U$ if and only if $(x, x^{-1}) \in G$.
\end{proposition}

\begin{proof}
If $x \in U$, then $(x, x^{-1}) \in H_1 \cap H_3 \subseteq G$ by \cref{isomorphicU}.

Conversely, suppose $(x, x^{-1}) \in G$. Since
\[
G \cap (\Omega \times \Omega) = H_3 = (H_1 \cap H_3)\cdot\llang a_3 \rrang,
\]
we may write
\[
(x, x^{-1}) = (\beta, \beta^{-1}) a_3^k
\]
for some $\beta \in U$ and $k \in \{0,1,2,3\}$. This gives
\[
x = \beta a_2^k \text{ and } x^{-1} = \beta^{-1} a_3^k.
\]
If $k \ne 0$, then $a_3$ and $a_2^{-1}$ or ($a_3^2$ and $a_2^2$) would be conjugate, which is impossible. Hence $k=0$ and $x = \beta \in U$.
\end{proof}

\begin{lemma}\label{lem:a2conj}
For every $g \in G_n$, there exists $h \in G_n$ such that
\[
a_1\restr{T_n} g a_3\restr{T_n}^{-1} g^{-1} = h a_2\restr{T_n} h^{-1}.
\]
\end{lemma}

\begin{proof}
The element $a_1 \restr{T_n}g a_3^{-1} \restr{T_n}g^{-1}$ lies in the coset of $H_{1,n} \cap H_{3,n}$ containing $a_2\restr{T_n} = a_1 a_3^{-1}\restr{T_n}$. By \cref{prop:centralizer}, the conjugacy class of $a_2\restr{T_n}$ is exactly this coset. Since $H_{1,n} \cap H_{3,n}$ is normal in $G_n$, the claim follows.
\end{proof}


    \begin{proof}[Proof of \cref{ggeom is conjugate to group generated by a_i}]
Let $b_1, b_2, b_3 \in \Omega_n$ be such that $b_i \sim a_i \restr{T_n}$ for all $i$ and $b_1 b_2 b_3 = \id$. 
We will show that there exist $\beta \in \Omega$ and $g_1, g_2, g_3 \in \llang a_1, a_2, a_3 \rrang$ such that
\[
b_i = \beta\, g_i a_i g_i^{-1} \beta^{-1}
\quad \text{for all } i=1,2,3.
\]

Once this is established, we may use the algebraic fact (see \cite[Lemma 1.3.2]{Pinkpolyn}) that replacing generators of a pro-$p$ group by conjugates does not change the generated group, and hence
\[
\llang g_i a_i g_i^{-1} \mid i=1,2,3 \rrang = \llang a_1, a_2, a_3 \rrang.
\]

Since we seek a common conjugating element, we may assume without loss of generality that $b_1 = a_1$. Then there exist $u,v \in G$ and $\tau \in \langle \sigma \rangle$ such that
\[
b_3 = (u,v)\tau (a_2,a_3) \tau (u^{-1}, v^{-1}).
\]
We may assume $\tau = \id$, since the case $\tau = \sigma$ can be treated similarly. Thus
\[
b_3 = (u,v)(a_2,a_3)(u^{-1}, v^{-1}).
\]
Using the relation $b_1 b_2 b_3 = \id$, we obtain
\[
b_2 = (v a_3^{-1} v^{-1},\, u a_2^{-1} u^{-1})\sigma \sim a_2.
\]

We proceed by induction on $n$. The base case is straightforward. For the induction step, assume that for any $c_1, c_2, c_3 \in \Omega_{n-1}$ satisfying
\begin{enumerate}[label=(\arabic*)]
    \item $c_i \sim a_i|_{T_{n-1}}$ for all $i$,
    \item $c_1 c_2 c_3 = \id$,
\end{enumerate}
there exist $\gamma \in \Omega_{n-1}$ and $h_1, h_2, h_3 \in G_{n-1}$ such that
\[
c_i = \gamma\, h_i\, a_i|_{T_{n-1}}\, h_i^{-1} \gamma^{-1}.
\]

Set
\begin{align*}
c_3 &\coloneqq v a_3 v^{-1} \sim a_3, \\
c_2 &\coloneqq u a_2 u^{-1} \sim a_2, \\
c_1 &\coloneqq (c_2 c_3)^{-1}.
\end{align*}
We verify that $c_1 \sim a_1$. Since
\[
b_3 = (c_2, c_3), 
\quad 
b_2 = (c_3^{-1}, c_2^{-1})\sigma \sim a_2 = (a_3^{-1}, a_2^{-1})\sigma,
\]
it follows that
\[
c_1 = c_3^{-1} c_2^{-1} \sim a_3^{-1} a_2^{-1} = a_1.
\]
Hence the induction hypothesis applies.

We obtain $\gamma \in \Omega_{n-1}$ and $h_1, h_2, h_3 \in G_{n-1}$. Define $\beta \coloneqq (\gamma, \gamma) \in \Omega_n$. Since $(\alpha,\alpha)$ centralizes $a_1$ for all $\alpha \in \Omega$, we have
\[
b_1 = a_1 = \beta a_1 \beta^{-1}.
\]

For $b_3$, we compute
\begin{align*}
b_3
&= (c_2, c_3) \\
&= (\gamma h_2 a_2|_{T_{n-1}} h_2^{-1} \gamma^{-1},\,
     \gamma h_3 a_3|_{T_{n-1}} h_3^{-1} \gamma^{-1}) \\
&= (\gamma, \gamma)
   (h_2 a_2 h_2^{-1},\, h_3 a_3 h_3^{-1})
   (\gamma^{-1}, \gamma^{-1}).
\end{align*}

Similarly,
\begin{align*}
b_2
&= (c_3^{-1}, c_2^{-1})\sigma \\
&= (\gamma, \gamma)
   (h_2 a_2^{-1} h_2^{-1},\, h_3 a_3^{-1} h_3^{-1})
   (\gamma^{-1}, \gamma^{-1})\sigma.
\end{align*}

Thus, replacing $b_i$ by $(\gamma,\gamma)^{-1} b_i (\gamma,\gamma)$, we may assume
\[
b_1 = a_1,
\quad
b_3 = (h_2 a_2 h_2^{-1},\, h_3 a_3 h_3^{-1}),
\quad
b_2 = b_1 b_3^{-1}.
\]

By \cref{index of U}, we have $G = U \cdot \llang a_3 \rrang$. Applying \cref{characterization of U} at level $n$, there exist $(x,y) \in G_n$ and $k \in \{0,1,2,3\}$ such that
\[
h_2^{-1} h_3 = (x y^{-1}) a_3^k.
\]

Set $w \coloneqq h_2 x$. Then $h_2 = w x^{-1}$ and $h_3 = w y^{-1} a_3^k$. Ignoring level restrictions for readability, we compute
\[
h_2 a_2 h_2^{-1} = w x^{-1} a_2 x w^{-1},
\quad
h_3 a_3 h_3^{-1} = w y^{-1} a_3 y w^{-1}.
\]

Hence
\begin{align*}
b_3
&= (w x^{-1} a_2 x w^{-1},\, w y^{-1} a_3 y w^{-1}) \\
&= (w,w)(x^{-1}, y^{-1}) a_3 (x,y) (w,w)^{-1}.
\end{align*}

Since $(x,y) \in G_n$ and $(w,w) \in \Omega_n$, we obtain
\[
b_2 = (w,w) a_1 (x^{-1}, y^{-1}) a_3^{-1} (x,y) (w,w)^{-1}.
\]
By \cref{lem:a2conj}, there exists $\beta \in G_n$ such that
\[
b_2 = (w,w)\, \beta a_2 \beta^{-1} (w,w)^{-1}.
\]
This completes the proof.
\end{proof}


\section{Discriminant calculations} \label{sec:discriminant}
    Let $f(x)$ be a rational map in $k(x)$.  Then, we write $f(x)=g(x)/h(x)$ where $g,h \in k[x]$ are relatively prime as polynomials in $k[x]$.  We define the discriminant of a rational function $g(x)/h(x)-t$ as the discriminant of the polynomial $g(x)-th(x)$ viewed as a polynomial over $k(t)$. 
    In other words,
\[\Delta_x(f(x)-t):=\Delta_x(g(x)-th(x)).\]

In this section, we would like to find a formula for the discriminant of the iterates of the rational function $f(x)=2/(x-1)^2$. 

Now, let $g_1(x)=2,\; h_1(x)=(x-1)^2$, then $f(x)=\frac{g_1(x)}{h_1(x)}$.  Suppose $f^{n}(x)=\frac{g_n(x)}{h_n(x)}$ for $n\geq 1$.  Then for $n\geq 2$, 
\[
 f^{n}(x)=f(f^{n-1}(x))=\frac{2}{(f^{n-1}(x)-1)^2}=\frac{2h_{n-1}^2(x)}{(g_{n-1}(x)-h_{n-1}(x))^2}.
\]

Hence, 
\begin{equation}\label{defn of gn and hn}
g_n(x)=2h_{n-1}^2(x) \text{ and } h_n(x)=(g_{n-1}(x)-h_{n-1}(x))^2.
\end{equation}

The discriminant formula for the iterates of the rational functions are given in \cite[proof of Proposition~1]{CH} by

\begin{equation}\label{discriminant formula}
\Delta_n:=\Delta_x(g_n(x)-th_n(x))=\pm \frac{l_n^{\epsilon_n+m_n-q_n-2}D_n^{m_n}}{l(h_n)^{m_n-\delta_n}\Res(g_n,h_n)}\prod_{r\in \cR_{f^{n}}}(g_n(r)-th_n(r))^{m_r},
\end{equation}

where $l_n=l_x(g_n(x)-th_n(x))$ and $l_x$ denotes the leading coefficient as a polynomial of $x$.  Furthermore, $$m_n=\deg_x (g_n(x)-th_n(x)),
\; \; D_n=l(h_n(x)g_n'(x)-g_n(x)h_n'(x)), \; \; $$ and $ \epsilon_n=\deg (h_n(x)), \; \delta_n=\deg (g_n(x)), \; q_n=\deg(h_n(x)g_n'(x)-g_n(x)h_n'(x))$.  Finally, let

 \[
   \cR_{f^{n}}=\{r\in\overline{k} :(h_ng_n'-g_nh_n')(r)=0\} 
 \]
 
be the set of ramification points of $f^{n}$  
and $m_r$ be the multiplicity of $r \in \cR_{f^{n}}$.

Note that $\cR_f=\{r\in \overline{k}: -4(r-1)=0\}=\{1\}$ with $m_1=1$ and the branch points  of $f^{n}$ are $\{0,2,\infty\}$. Let $n\geq 2$ and $r \in \cR_{f^n}$. If  $f^n(r)=0$, we have $m_r=3$. Similarly, if $f^n(r)=\infty$, then $m_r=1$. If $f^n(r)=2$, then both multiplicities can occur.

Our main goal in this section is to prove ~\cref{discriminant}.
\begin{proposition}\label{discriminant}
	For $n\geq 1$, the discriminant can be calculated as $\Delta_n=c_nt^{a_n}(2-t)^{b_n}$ where $a_n,b_n$ are nonnegative integers and $c_n$ is a power of $2$ up to sign.
\end{proposition}
\begin{proof}
	Using  \cref{discriminant formula}, the proof follows from combining \cref{leading_coeff}, \cref{resultants}, \cref{ramificationpoints} as well as \cref{Dn}.
\end{proof}

First of all, we would like to find the finite primes of $\overline{k}(t)$ dividing $\Delta_n$ and these terms can come from $\prod_{r\in \cR_{f^{n}}}(g_n(r)-th_n(r))^{m_r}$ and $l_n$.		
We begin with analyzing $l_n$ given in ~\cref{discriminant formula}.

\begin{lemma}\label{leading_coeff}
For all $n \geq 2$, we have $\deg g_n = \deg h_n = 2^n$, $l (g_n)=2$ and $l(h_n)=1$. Moreover, $l_n$ is equal to $(2-t)$. 
\end{lemma}
\begin{proof}
Recall that we have $g_n(x)=2h_{n-1}^2(x)$ and $h_{n}(x)=(g_{n-1}(x)-h_{n-1}(x))^2$. For $n=2$ we obtain $g_2(x)=2h_1(x)^2=2(x-1)^4$, and so $\deg g_2(x)=4=2^2$ and $l(g_2)=2$.
Additionally, we can write \[h_2(x)=(g_1(x)-h_1(x))^2=(2-(x-1)^2)^2. \] 
Observe that the degree of  $h_2(x)$ equals $2^2$ and $l(h_2)=1$.

Let $n\geq 3$ and assume the statement is true for $n-1$.
Then $\deg g_n=2\deg h_{n-1}=2\cdot2^{n-1}=2^n$ since $\deg h_{n-1}=2^{n-1}$ by assumption.
Moreover, we have $l(g_n)=2l(h_{n-1})^2=2\cdot1=2$. On the other hand, we get $\deg h_n(x)=2\deg(g_{n-1}-h_{n-1})=2\cdot2^{n-1}=2^n$ because both $g_{n-1}$ and $h_{n-1}$ are of same degree with $l(g_{n-1})=2$ and $l(h_{n-1})=1$. Similarly, we deduce $l(h_n)=l(g_{n-1}-h_{n-1})^2=1$ as $\deg g_{n-1}=\deg h_{n-1}$ with $l(g_{n-1})=2$ and $l(h_{n-1})=1$.

Finally, for all $n\geq 1$ we obtain $l_n=l_x(g_n(x)-th_n(x))=(2-t)$ by using $\deg(g_n)=\deg(h_n)$ and $l(g_n)=2$, $l(h_n)=1$. 
\end{proof}

We next show that the resultant of $g_n$ and $h_n$ for any $n\geq 3$ is a power of $2$. The second statement is needed for the proof of \cref{ramificationpoints}.

\begin{lemma} \label{resultants}
Let $n\geq 3$ and let $2 \leq k \leq n$.
Then $\Res(g_{k},h_n)$ is a power of $2$. 
\end{lemma}
\begin{proof} We will prove the assertion by induction.  First, assume $n=3$, then $k=2$ or $k=3$.  In this case, we see that $\Res(g_2,h_3)$ and $\Res(g_3,h_3)$ are powers of $2$ by a calculation on MAGMA.\\
	
	Now assume for induction that $\Res(g_k,h_m)$ is a power of $2$  for all $(k,m)$ such that $2\leq k \leq m$ and $3\leq m < n$.
	
	\textbf{Claim 1 : } We will show that $\Res(g_{n-1},h_n)$ and $\Res(g_n,h_n)$ are powers of $2$.  From the definitions of the resultant, $h_n$ and $g_n$ we see that
	
	 \begin{align*}
	\displaystyle \Res(g_{n-1},h_n)&=2^{\deg h_n}\prod_{g_{n-1}(r)=0}h_n(r) \;\text{where the product is taken over all the roots,} \\
	 &=2^{\deg h_n}\prod_{g_{n-1}(r)=0}(g_{n-1}(r)-h_{n-1}(r))^2 \; \text{ by } \eqref{defn of gn and hn},\\
	& =\displaystyle 2^{\deg h_n}\left(\frac{1}{2^{\deg h_{n-1}}}\Res(g_{n-1},h_{n-1})\right)^2\\
	&= \displaystyle (\Res(g_{n-1},h_{n-1}))^2
	\end{align*}
	
and the last term is a power of $2$ by induction.  Moreover,
     \begin{align*}
    \displaystyle \Res(g_n,h_n)&=2^{\deg h_n}\prod_{g_n(r)=0} h_n(r)\\
    &\displaystyle=2^{\deg h_n}\prod_{g_n(r)=0} (g_{n-1}(r))^2\;\;\text{since}\;g_n(r)=0\;\text{implies}\;h_{n-1}(r)=0,\\
    &\displaystyle=2^{\deg h_n}\prod_{h_{n-1}(r)=0} g_{n-1}^4(r)\\
    &\displaystyle =\displaystyle 2^{\deg h_n}\left(\frac{1}{l(h_{n-1})^{\deg g_{n-1}}}\Res(g_{n-1},h_{n-1})\right)^4\\
    &\displaystyle= 2^{\deg h_n}\Res^4(g_{n-1},h_{n-1}) \;\;\text{since}\;  l(h_{n-1})=1
    \end{align*}
    where the third line above follows from the fact if $g_n(r)=0$ then $f^n(r)=0$ i.e. $f(f^{n-1}(r))=0$ which implies $f^{n-1}(r)=\infty$ and $h_{n-1}(r)=0$, and also every such $r$ is a double root of $g_n(x)$.  Since the resultant in the last line is a power of $2$ by the induction assumption, we are done.\\
    
    	\textbf{Claim 2 : } We will show that $\Res(g_k,h_n)$ is a power of $2$ for $2\leq k\leq n-2$.
    	
    	 \begin{align*}
           \displaystyle \Res(g_k,h_n)&=2^{\deg h_n}\prod_{g_k(r)=0} h_n(r)\\
           &\displaystyle =2^{\deg h_n}\prod_{g_k(r)=0}(g_{n-1}(r)-h_{n-1}(r))^2\\
           &\displaystyle =2^{\deg h_n}\prod_{g_k(r)=0}(h_{n-1}(r))^2\\
           &\displaystyle =\displaystyle 2^{\deg h_n}\left(\frac{1}{2^{\deg h_{n-1}}}\Res(g_k,h_{n-1})\right)^2\\
           &\displaystyle= \Res^2(g_k,h_{n-1}) \;\;\text{since}\;  \deg h_n=2\deg h_{n-1} \;\;\text{for}\; n \geq 2.
    	\end{align*}
    	In the third line above we use the fact if $g_k(r)=0$ then $f^k(r)=0$ and hence $f^m(r)=2$ for $m>k$ which implies $g_m(r)=2h_m(r)$  for $m>k$ .  Note that $k<n-1$ here.  Since the resultant in the last line is a power of $2$ by the induction assumption, we are done.\\
    	
\end{proof}

\begin{lemma}\label{ramificationpoints}
	For  any $n\geq 2$ and any $r \in \cR_{f^{n}}$,\; $\displaystyle\prod_{r\in \cR_{f^{n}}}(g_n(r)-th_n(r))^{m_r} = a_n t^{b_n}(t-2)^{c_n}$, where $b_n$ and $c_n$ are positive integers, and $a_n$ is a power of $2$ up to sign.
\end{lemma} 
\begin{proof} 
	 When $n=2$, we compute that the statement holds. Assume for the rest of the proof that $n\geq 3$. Since the only branch points of $f^{n}$ are $\{0,\infty,2\}$, $f^{n}(r)=\displaystyle\frac{g_n(r)}{h_n(r)}=0,2,\mbox{ or } \infty$ for any $r\in \cR_{f^{n}}$.  

The idea of the proof is to write the product in the lemma as a product of three terms according to the branching data and then express and compute each product using resultant. We decompose this product into three terms: 
	 \begin{align*}
	  \displaystyle\prod_{r\in \cR_{f^{n}}}(g_n(r)-th_n(r))^{m_r}& =\displaystyle\prod_{f^n(r)=0}(g_n(r)-th_n(r))^{m_r} \displaystyle\prod_{f^n(r)=\infty}(g_n(r)-th_n(r))^{m_r} \displaystyle\prod_{f^n(r)=2}(g_n(r)-th_n(r))^{m_r}  \\ 
	  &=\displaystyle\prod_{f^n(r)=0}(-t)^3h_n(r)^3 \displaystyle\prod_{f^n(r)=\infty}g_n(r) \displaystyle\prod_{\substack{f^n(r)=2 \\ r \in \cR(f^n) }}((2-t)h_n(r))^{m_r}.
	 \end{align*}
In this equation, we take each product over distinct $r \in \cR(f^n)$ satisfying the given equation in the corresponding term. In the case $f^n(r)=0$, we know that $r \in \cR(f^n)$ since $n\geq 3$.
Similarly, in the case $f^n(r)=\infty$, $r$ is in $\cR(f^n)$. However, in the third case we take the product over all $r \in \cR(f^n)$ such that $f^n(r)=2$. Since $f^{-n}(r)$ contains $0,1,2,\infty $ (with multiplicity) and $\{0,2\}$ is not contained in $\cR(f^n)$, we only take $r=1$ in this term.
	 
We use here the fact that $m_r=3$ when $f^n(r)=0$ and $m_r=1$ when $f^n(r)=\infty$. We will calculate the constants up to sign arising from these terms. The first term gives us 
$\displaystyle\prod_{g_n(r)=0}h_n(r)^3$, which is up to sign equal to the $3/4$th power of $\Res(g_n,h_n)$ divided by  $l(g_n)^{\deg(h_n)}$. By \cref{leading_coeff} and \cref{resultants}, it is a rational power of $2$ up to sign. 

Similarly, the second term gives the product $\displaystyle\prod_{h_n(r)=0}g_n(r)$. Again, it follows from \cref{leading_coeff} and \cref{resultants} that $\displaystyle\prod_{h_n(r)=0}g_n(r)$ is the square-root of an integer that is a power of $2$.

For the last term, we proceed as follows. We first observe that since $2$ is not a branch point of $f$, if $r \in \cR(f^n)$, then $f^{k}(r)=0$ for some $2\leq k\leq n-1$. Hence, its multiplicity is always $3$.

\begin{align} \displaystyle\prod_{\substack{f^n(r)=2 \\ r \in \cR(f^n) }}(h_n(r))^{m_r} & = \displaystyle\prod_{2\leq k\leq n-1} \displaystyle\prod_{f^{k}(r)=0}(h_n(r))^{3}  .
 \end{align}                                                   

We note that the inner product here is taken over the distinct roots of $f^k$. We see that for any $k \in \{2,\ldots,n-1\}$, 
\begin{equation}\label{two term} \displaystyle\prod_{f^{k}(r)=0}h_n(r)^3 =\displaystyle\prod_{g_k(r)=0}h_{n}(r)^{3}=\Res(g_k,h_n)^{3/4}/l(g_k)^{3/4\deg(h_n)}.
\end{equation}

Since, $a_n$ is an integer and it is a rational power of $2$, it is $2^k$ for some $k\geq 0$.
\end{proof}

\begin{lemma}\label{Dn}
For all $n\geq 2$, $D_n=(-4)^n$.
\end{lemma}

\begin{proof}
We will show that $D_n=\pm 4D_{n-1}$ for $n\geq 3$ and calculate that $D_2=-4$. 

First, let us express the polynomial $h_ng_n' - g_nh_n'$ in terms of $g_i$ and $h_i$ with $1\leq i\leq n$. Recall that $g_n=2h_{n-1}^2$ and $h_n = (g_{n-1} - h_{n-1})^2$.
	
	\begin{align}\label{criticalpolynomial}
		h_ng_n' - g_nh_n' & =4(g_{n-1}-h_{n-1})^2 h_{n-1} h_{n-1}'-4 h_{n-1}^2(g_{n-1}-h_{n-1})(g_{n-1}'-h_{n-1}') \\ \nonumber
		& =4(g_{n-1}-h_{n-1}) h_{n-1}\left[(g_{n-1}-h_{n-1}) h_{n-1}'-h_{n-1}(g_{n-1}'-h_{n-1}')\right] \\ \nonumber
		& =4(g_{n-1}-h_{n-1}) h_{n-1}\left[g_{n-1} h_{n-1}'-h_{n-1} g_{n-1}'\right]\\ \nonumber
		& =4(g_{n-1}-h_{n-1})(g_{n-2} - h_{n-2})^2\left[g_{n-1} h_{n-1}'-h_{n-1} g_{n-1}'.\right]
	\end{align}
 We can check that $h_1g_1' - g_1h_1' = 4(x-1)$ and $$h_2g_2' - g_2h_2'=4^2(g_1-h_1)(x-1)^3=-4^2(x^2-2x-1)(x-1)^3.$$
Hence $D_2=-4$. Assume for induction that $D_{n-1}=(-4)^{n+1}$ for $n\geq 2$. By \cref{leading_coeff}, we have
	\begin{align*}
		D_n&=-4 D_{n-1} l(g_{n-1}-h_{n-1}) l(g_{n-2}-h_{n-2}) \\
		      &=(-4)^n.
	\end{align*}
\end{proof}


\section{Proof of the Main Theorem}\label{sec:proof}

\subsection{Field of constants}

\noindent Let $K$ denote the field $k(t)$ and $K_n$ be the splitting field of $f^n(x)-t$ over $K$ for $n\geq 1$.
Define $K_\infty:= \bigcup_n K_n$.
Let  $\alpha$ be any vertex in the tree $T$.  We denote the two vertices of $f^{-1}(\alpha)$ by  $\alpha_1, \alpha_2$ where $\alpha_1 = 1 + \sqrt{2/\alpha}$ and  $\alpha_2 = 1 - \sqrt{2/\alpha}$.
Inductively, for a word $ \ell$ consisting of letters $1,2$ and $\alpha_ \ell$ in the preimage tree of $\alpha$ we set 
\[\alpha_{ \ell 1}= 1 + \sqrt{\frac{2}{a_ \ell}} \text{ and } \alpha_{ \ell 2}= 1 - \sqrt{\frac{2}{a_ \ell}}\]
such that we have $f^{-1}(\alpha_ \ell) = \{\alpha_{ \ell 1}, \alpha_{ \ell 2}\}.$
To get an overview, here are the first three levels of the preimage tree with our labeling:

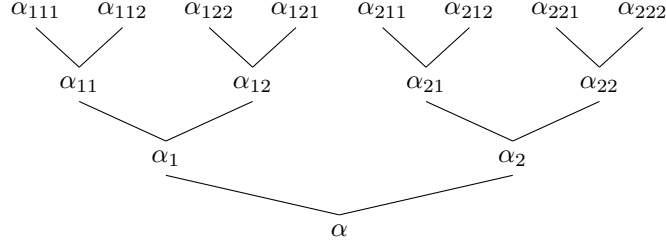
\begin{center}
	\scalebox{1}{	
		\begin{forest}
			my tree
			[$\alpha$
			[$\alpha_{2}$
			[$\alpha_{22}$
			[$\alpha_{222}$]
			[$\alpha_{221}$]
			]
			[$\alpha_{21}$
			[$\alpha_{212}$]
			[$\alpha_{211}$
			]]]
			[$\alpha_{1}$
			[$\alpha_{12}$
			[$\alpha_{121}$]
			[$\alpha_{122}$]
			]
			[$\alpha_{11}$
			[$\alpha_{112}$]
			[$\alpha_{111}$
			]]]
			]
	\end{forest}}
\end{center}

The following lemma captures some basic rules.
\begin{lemma}\label{basic properties of root calculation}
	For any words $ \ell$, $ \ell'$ consisting of letters $1,2$ and of the same length or $\alpha_ \ell, \alpha_{ \ell'} \in f^{-n}(\alpha)$ for some $n$, we have
	\begin{enumerate}
		\item $\displaystyle\alpha_{ \ell1} \alpha_{ \ell2} = 1- \frac{2}{\alpha_ \ell} = \frac{\alpha_ \ell -2}{\alpha_ \ell}.$
		\item $\displaystyle\bigl((\alpha_{ \ell1} -1) (\alpha_{ \ell'1} -1)\bigr)^2 = \frac{4}{\alpha_ \ell \alpha_{ \ell'}}. $
	\end{enumerate}
\end{lemma}
\begin{proof}
	Using the definition, we calculate
	\[\alpha_{ \ell1} \alpha_{ \ell2} 
	= \left(1+ \sqrt{\frac{2}{\alpha_ \ell}}\right)\left(1- \sqrt{\frac{2}{\alpha_ \ell}}\right)
	= 1 - \frac{2}{\alpha_ \ell} 
	= \frac{\alpha_ \ell -2}{\alpha_ \ell}.\]
	Furthermore, we get
	\begin{align*}
		\bigl((\alpha_{ \ell 1} -1) (\alpha_{ \ell' 1} -1)\bigr)^2
		= \left( \sqrt{\frac{2}{\alpha_ \ell}} \sqrt{\frac{2}{\alpha_{ \ell '}}}\right)^2 
		= \frac{4}{\alpha_ \ell \alpha_{ \ell'}}.
	\end{align*}
\end{proof}

\begin{lemma}\label{property of i}
	
	The imaginary number $i$ is contained in $K_3$.
\end{lemma}
\begin{proof}
	Using \cref{basic properties of root calculation}, we can calculate that $-1$ is a square in $K_3 = K(\alpha_ \ell \mid \text{length}( \ell) \leq 3):$
	\begin{align*}
		\left( \frac{(\alpha_{111} -1) (\alpha_{121}-1)}{2} \cdot \frac{(\alpha_{11}-1)}{(\alpha_{21}-1)}\right)^2
		&= \frac{4}{\alpha_{11}\alpha_{12}} \cdot \frac{1}{4} \cdot \frac{(\sqrt{2/\alpha_1})^2}{(\sqrt{2/\alpha_2})^2} \\
		&= \frac{1}{\alpha_{11}\alpha_{12}} \cdot \frac{\alpha_2}{\alpha_1} \\
		&= \frac{\alpha_1}{\alpha_1-2} \cdot \frac{\alpha_2}{\alpha_1} \\
		&= \frac{\alpha_2}{\alpha_1-2} \\
		&= \frac{1 - \sqrt{2/\alpha}}{-1 +\sqrt{2/\alpha}} \\
		&= -1.
	\end{align*}
\end{proof}
\begin{lemma}\label{property of squareroottwo}
	
	The number $\sqrt{2}$ is contained in $K_4$.
\end{lemma}
\begin{proof}
We verified this calculation on MAGMA. \footnote{The code can be found here \url{https://github.com/zofiagoaska/IteratedMonodromyGroups}}. 
See also \cite[ Proof of Theorem 3.8]{foz}.
\end{proof}

\begin{lemma}\label{properties of roots}
	
	For any vertex $\alpha$ of $T$, we have $\sqrt{\alpha/2}\in K_{\infty}$ and 
	\[2(\alpha-2)=\left[\frac{1}{(\alpha_1-1)}\frac{2}{(\alpha_{11}-1)}\frac{2}{(\alpha_{21}-1)} \right]^2.\]
\end{lemma}
\begin{proof}
	The first claim follows from the fact that $\alpha_i-1 \in K_{\infty}$ and $\alpha/2=\left(\frac{1}{\alpha_1-1}\right)^2.$
	By using the definition of $\alpha_1$ and \cref{basic properties of root calculation}, we obtain
	\begin{align*}
		\left[\frac{1}{(\alpha_1-1)}\frac{2}{(\alpha_{11}-1)}\frac{2}{(\alpha_{21}-1)}\right]^2
		&= \frac{1}{(\sqrt{2/\alpha})^2}\cdot\frac{\alpha_1\alpha_2}{4}\cdot4^2  \\
		&=  \frac{\alpha}{2}\cdot \frac{\alpha-2}{\alpha}\cdot 4\\
		&= 2(\alpha-2).
	\end{align*}
\end{proof}
\begin{proposition}\label{K-infty}
	The field $K_{\infty}$ (in fact $K_4$) contains the field $L=K(i, \sqrt2, \sqrt{t},\sqrt{t-2})$.
\end{proposition}
\begin{proof}
	Putting together \cref{property of i}, \cref{property of squareroottwo} and \cref{properties of roots} shows that $i$, $\sqrt{2}$, $\sqrt{t} $ and $ \sqrt{t-2}$ are in \linebreak
	$K_4 \subseteq K_{\infty}$.
\end{proof}

\subsection{The Frattini subgroup of $\ga$}

We denote the Frattini subgroup of $\ga$ by $\Phi(\ga)$.  

\begin{theorem}\label{Frattinithm}
The Frattini subgroup $\Phi(\ga)$ of $\ga$ is the subgroup of $\ga$ fixing the field $$L=K(i, \sqrt{2}, \sqrt{t}, \sqrt{2-t}).$$
\end{theorem}

\begin{proof}
Let $H$ be a maximal subgroup of $\ga= \Gal(K_{\infty}/k(t))$. Since $\ga$ is a $2$-group, the index of $H$ in $\ga$ is two. 
Let $K_H$ denote its fixed field in $K_{\infty}$. Because the field extension $K_H/K$ has degree two, 
we can write $K_H=K(\sqrt{a(t)})$ for a square-free polynomial $a(t)$ in $k[t]$. 

 Since any prime in $k[t]$ dividing $a(t)$ ramifies in $K_H$, it ramifies in $K_{\infty}$. Hence, $a(t)$ ramifies in $K_H$, and so does in $K_n$ for some $n \geq 1$. 
 Consequently, $a(t)$ divides $\Delta_n$ in $k[t]$.  We calculated in \cref{discriminant} that the discriminant $\Delta_n$ is of the form $c_nt^{a_n}(2-t)^{b_n}$ where $c_n$ is a power of $2$ and $a_n, b_n \in \ZZ^{\geq 0}$.
 Therefore, we can write $a(t)=ct^i(2-t)^j$ with $c \in k$ and $i, j \in \{0,1\}$.

Specializing at $t=1$, we know $a(1)$ ramifies in $k_n=k(f^n(x)-1)$ and $\Delta_{k_n}=\pm 2^{m}$ for some $m \in\ZZ^{\geq 0}$. 
It follows that we can write $a(t)=\pm 2^{m}t^i(2-t)^j$ where $m \in \ZZ^{\geq 0}$ and $i,j \in \{0,1\}$. 
This implies that $K_H=K(\sqrt{a(t)})$  is contained in $ L=K(i, \sqrt{2}, \sqrt{t}, \sqrt{2-t})$.

To sum up, $\Gal(K_{\infty}/L)$ is a subgroup of $H$ for any maximal subgroup $H$ of $\ga$, which implies  $\Gal(K_{\infty}/L)$ is a subgroup of $\Phi(\ga)$. \\
On the other hand, we can describe the Galois group of $K_{\infty}$ over $L$ as 

$$ \Gal(K_{\infty}/L) 
= \Gal\bigl(K_{\infty}/K(i)\bigr) \cap \Gal\bigl(K_{\infty}/K(\sqrt{2})\bigr) \cap \Gal\bigl(K_{\infty}/K(\sqrt{t})\bigr) \cap \Gal\bigl(K_{\infty}/K(\sqrt{2-t})\bigr),$$
where each group on the right hand side is an index $2$ subgroup of $\ga$\footnote{We assume here that $K$ does not contain $i$ or $\sqrt{2}$. Otherwise,  for $\alpha \in \{i,\sqrt{2}\}$, $\Gal(K_{\infty}/K(\alpha)) $ is not an index two subgroup of $\ga$. However, in this case the corresponding field is already contained in $K$ and the intersection still gives the Galois group of $K_{\infty}$ over $L$.}, 
and therefore a maximal subgroup of $\ga$. 
This shows that $\Phi(\ga)$ is contained in $\Gal(K_{\infty}/L)$ which concludes proving $\Phi(\ga)=\Gal(K_{\infty}/L)$.
\end{proof}

We next turn to the proof of our main theorem. Let $k$ be a number field and let $a \in k\backslash\{0,2\}$. 

\begin{theorem}
The arboreal Galois group $G_a(f)$ equals the arithmetic iterated monodromy group $\ga(f)$ if and only if they agree on the fourth level i.e. $G_{4,a}(f) = \ga_4(f)$.	
\end{theorem}

\begin{proof}
\cite[Theorem~1.3]{BGJT25s} asserts the existence of a natural number $m \geq 1$ such that $G_{m,a}=\ga_{m}$ implies $G_{a}=\ga$.
More precisely, the authors of \cite{BGJT25s} show that one can take $m$ to be the smallest level for which the fixed field of the Frattini subgroup is contained in $\ga_m$.
In our situation, \cref{Frattinithm} together with ~\cref{K-infty} imply that $m$ can be as small as $4$.
\end{proof}


\section{The normalizer of $G$ in $\Omega$}\label{Norm of G}

Let $N$ denote the normalizer of $G$ in $\Omega$. We note that $N$ is a closed subgroup of $\Omega$ as well. Since $G$ is normal in $\ga$, we will assume that $\ga$ is contained in $N$.

\begin{lemma}\label{lem:arU} If $(x,y)\tau \in N$ for some $x,y \in \Omega$ and $\tau \in S_2$, then 
$x$ and $y$ are both in $N$ and $xy^{-1}$ is in $U$.
\end{lemma}
\begin{proof}
Let $\beta=(x,y)\tau \in N$. Since $a_1 = \sigma \in G$, we remark that whenever $(x,y)\sigma \in \ga$, $(x,y)$ is in $\ga$ as well. 
We have $a_3=(a_2,a_3)$ and $a_1a_3a_1^{-1}=(a_3,a_2)$ in $G$. Since $\beta \in N$, the conjugate $(x,y)(a_2,a_3)(x,y)^{-1} \in G$. By the self-similarity of $G$, we have $xa_2x^{-1} \in N$. Similarly conjugating $(a_3,a_2)$ by $\beta$, we find that $xa_3x^{-1} \in G$. Hence $x$ is in $N$. Repeating this argument for $y$, we have $y$ is also in $N$.

Conjugating  $a_1$ by $\beta$, we obtain that
\[
(x,y)\sigma(x,y)^{-1} = (xy^{-1}, yx^{-1})\sigma ,
\]
is in $G$ and hence, $xy^{-1}$ is in $G$ again by self-similarity. Moreover, \cref{inverse pair} implies that $xy^{-1} \in U$.
\end{proof}

We next show that there are no odometers in $\ga$.

\begin{proposition}\label{Gaodometer}
The group $N$ does not contain any odometers.
\end{proposition}
\begin{proof}
Let $\alpha \in N$ be an odometer. Then $\alpha=(\alpha_0,\alpha_1)\sigma$ for some $\alpha_0,\alpha_1 \in \Omega$. 
By ~\cref{pink:odometer}, the sign of an odometer has to equal $-1$ on each level i.e. $\sgn_n(\alpha)=-1$ for all $n\geq 1$. 
Using ~\cref{conjugation rules in Aut T}, $\alpha$ is conjugate to $(\alpha_0\alpha_1,\id)\sigma$  and hence
$$ \sgn_n(\alpha)=\sgn_{n-1}(\alpha_0\alpha_1)$$ for all $n\geq 2$. 

By ~\cref{pink:odometer}, $\alpha_0\alpha_1$ is also an odometer. However, by Lemma~\ref{lem:arU}, $\alpha_0\alpha_1^{-1} \in U$, and since $G$ does not contain any odometers, there is some $m\geq 1$ for which $$\sgn_m(\alpha_0\alpha_1^{-1})=\sgn_m(\alpha_0\alpha_1)=1.$$

We arrive at a contradiction.

\end{proof}

\begin{proposition}\label{U torsion free}
 The subgroup $U$ is a free $\ZZ_2[i]$-module of rank one. In particular, every element in $U$ has infinite order. 
\end{proposition}
\begin{proof}
Since $U$ is abelian and closed in $G$, it is a $\ZZ_2$-module.  We will view $U$ as a module over $\ZZ_2[i]$ by $i \cdot \alpha :=a_3 \cdot \alpha a_3^{-1}$. This allows us to view $G$ as a semidirect product of $\ZZ_2[i]$ and $\mu_4=\langle i \rangle$ where the action of $i$ is given by multiplication by $i$. We define a map
 \[ \phi:  \ZZ_2[i] \to U \] by \[ a+bi \mapsto \gamma_1^{b} \gamma_2^a .\]
 
 We check that $\phi$ is a $\ZZ_2[i]$-module homomorphism. Since $U$ is generated by $\gamma_1$ and $\gamma_2$, it is surjective. We first observe that every non-trivial quotient of $\ZZ_2[i]$ is finite. However, by \cref{orderG} and \cref{index of U}, $U$ is infinite. Hence, $\phi$ is an isomorphism of $\ZZ_2[i]$-modules.
 \end{proof}

\begin{proposition}\label{U infinite order}
The subgroup $U$ is the set of all elements of infinite order in $G$.
\end{proposition}

\begin{proof}
By \cref{U torsion free}, every element of $U$ has infinite order. Hence, we only need to show that the elements of $G$ that are outside of $U$ have finite order. In fact, we will show that their order is at most $4$. Since $G$ is isomorphic to the semidirect product of $U$ and $\langle a_3 \rangle$ by \cref{lem:U-semi}, any element $g \in G \backslash U$ is  $\alpha a_3^k$ for some $\alpha \in U$ and $k \in \{1,2,3\}$. If $g^2=\id$, then there is nothing to prove. Hence, we may assume that  $g^2=\alpha' a_3^2$ for some $\alpha' \in U$ and $$g^4=(\alpha' a_3^2)^2=\alpha' a_3^2 \alpha' a_3^{2}=\alpha' \alpha'^{-1} =\id. $$ 
\end{proof}

\begin{lemma}
The subgroups $\llbracket G, G \rrbracket$ and $U$ are normal in $N$.
\end{lemma}

\begin{proof}
 By \cref{U infinite order}, $U$ is the set of all elements of infinite order in $G$; hence, it is a characteristic subgroup of $G$. Since $\llbracket G, G \rrbracket$ is also characteristic in $G$ and $G$ is normal in $N$, it follows that $\llbracket G, G \rrbracket$ and $U$ are both normal in $N$.

\end{proof}

\begin{lemma}\label{stabilizer of a3 in U}
Let $n\geq 3$. The number of elements in $U_n$ that commute with $a_3\restr{T_n}$ is at most $2$.
\end{lemma}

\begin{proof}
We bound the size of $C_{G_n}(a_3\restr{T_n}) \cap U_n$. By \cref{prop:centralizer}, the centralizer $C_{G_n}(a_3\restr{T_n})$ has order $8$, and $\langle a_3\restr{T_n} \rangle \subseteq C_{G_n}(a_3\restr{T_n})$.

Consider the natural injection
\[
\langle a_3\restr{T_n} \rangle / (\langle a_3\restr{T_n} \rangle \cap U_n)
\hookrightarrow
C_{G_n}(a_3\restr{T_n}) / (C_{G_n}(a_3\restr{T_n}) \cap U_n).
\]
It therefore suffices to show that $\langle a_3\restr{T_n} \rangle \cap U_n = \{\id\}$.

Suppose for contradiction that $a_3^2\restr{T_n} \in U_n$. By \cref{isomorphicU}, this implies $(a_3^2\restr{T_{n}}, a_3^{-2}\restr{T_{n}}) \in G_{n+1}$. On the other hand, $a_3^2 = (a_3^2, a_2^2)$ also lies in $G$, contradicting \cref{cor:uniqueG}. Hence $a_3^2\restr{T_n} \notin U_n$, and thus
\[
\langle a_3\restr{T_n} \rangle \cap U_n = \{\id\}.
\]
It follows that $|C_{G_n}(a_3\restr{T_n}) \cap U_n| \le 2$.
\end{proof}

\begin{proposition}\label{G arithmetic order}
The order of $N_m$ is at most $4$ times the order of $N_{m-1}$.
\end{proposition}

\begin{proof}
Let $(x,y) \in N$. By the discussion above, we have $\rho = yx^{-1} \in U$, so we may write
\[
(x,y) = (x, \rho x)
\]
for some $\rho \in U$.

Suppose $\gamma = (x, \rho x)$ and $\gamma' = (x, \rho' x)$ are elements of $N$. Then
\[
\gamma a_3 \gamma^{-1}
=
(x a_2 x^{-1},\, \rho x a_3 (\rho x)^{-1})
\]
and
\[
\gamma' a_3 (\gamma')^{-1}
=
(x a_2 x^{-1},\, \rho' x a_3 (\rho' x)^{-1}).
\]

Since $N$ is the normalizer of $G$ in $\Omega$, both conjugates lie in $G$. By \cref{cor:uniqueG}, we obtain
\[
\rho x a_3 (\rho x)^{-1}
=
\rho' x a_3 (\rho' x)^{-1}.
\]
Hence $\rho'^{-1} \rho$ centralizes $x a_3 x^{-1}$. Since $U$ is normal in $N$, conjugation by $x$ induces a bijection between
\[
C_G(a_3) \cap U
\quad \text{and} \quad
C_G(x a_3 x^{-1}) \cap U.
\]
By \cref{stabilizer of a3 in U}, there are at most two such elements. Therefore, for each $x$ there are at most two possible choices of $\rho$, and hence
\[
|N_k \cap (\Omega \times \Omega)| \le 2\,|N_{k-1}|.
\]

Since $a_1 = \sigma \in \ga$, we have
\[
|N_m| = 2\,|N_m \cap (\Omega \times \Omega)|.
\]
Combining these inequalities yields
\[
|N_m| \le 4\,|N_{m-1}|,
\]
as claimed.
\end{proof}

\begin{corollary}\label{orderGa}
For $m=4$, the order of $N_4$ is equal to $2^8$ and for any $m \geq 4$, the order of $N_m$ is at most $2^{2m}$.
\end{corollary}

\begin{proof}
A computation in MAGMA shows that $|N_4| = 4^4$. The result now follows from \cref{G arithmetic order}.
\end{proof}

The Hausdorff dimension of a closed subgroup $H$ of $\Omega$ is defined as $\lim_n \frac{ \log(|H|)}{2^n-1}$. Since the arithmetic iterared monodromy group $\ga$ of $f$ is contained in $N$,  ~\cref{orderGa} implies the following:

\begin{corollary}\label{Hdimension}
The Hausdorff dimension of $\ga$ is zero.
\end{corollary}


\subsection*{Constructing some elements of $N$}
 We next exhibit explicit elements in $N$.  Recall that  the subgroup $U$ is topologically generated by $$\gamma_1 := a_2 a_3^{-1} \mbox{ and }\gamma_2 := a_3^{-1} a_2.$$
 
 Define the following elements in $\Omega$: 
$$x:=(\gamma_1^2x,x), \quad y:=(\gamma_2^2 y, y), \quad z:=(z a_3^{-1}, z a_3^{-1}).$$

\begin{proposition}\label{xyz in N}The elements $x,y, z$ we defined are all in $N$. Moreover,
\begin{enumerate}
\item The elements $x,y$ commute with $a_3$.
\item $za_3z^{-1}=a_3^{-1}$, $z\gamma_1 z^{-1}=\gamma_1$, and $z\gamma_2 z^{-1}=\gamma_2^{-1}$. 
\item $x a_2x^{-1}=[a_3,\gamma_1^2]a_2$ and $y a_2y^{-1}=[a_3,\gamma_2^2]a_2$, where $[g,h]$ denotes $ghg^{-1}h^{-1}$ for any $g,h \in \Omega$.
\end{enumerate}
\end{proposition}

\begin{proof}
We will first prove the first and third claims. We note that this proves that $x,y$ are elements of $N$. Furthermore, the second claim and the fact that $z$ commutes with $a_1$ imply that $z\in N$.

Let  $\alpha \in U$ and define $w:=(\alpha w,w) \in \Omega$. Notice that both $x$ and $y$ are in this form. 
Assume that $[a_3,\alpha]=[w,a_1]$. We will show by induction that for all $n\geq 1$, $w\restr{T_n}$ commutes with $a_3\restr{T_n}$ and $wa_2w^{-1}\restr{T_n}=[a_3,\alpha]a_2\restr{T_n}$. Then we show that $[a_3,\gamma_1^2]=[x,a_1]$  and $[a_3,\gamma_2^2]=[y,a_1]$.
 
Let $n\geq 1$. The statement holds for $n=1$ trivially.  Assume that $w\restr{T_{n-1}}$ commutes with $a_3\restr{T_{n-1}}$ and 
\begin{equation}\label{inductiona2} wa_2w^{-1}\restr{T_{n-1}}=[a_3,\alpha]a_2\restr{T_{n-1}}. \end{equation}
We first show that $wa_3w^{-1}=a_3$.

\begin{equation} wa_3w^{-1}=(\alpha w,w)(a_2,a_3)(w^{-1}\alpha^{-1}, w^{-1}) 
                                              =(\alpha wa_2w^{-1}\alpha^{-1}, wa_3w^{-1}).
 \end{equation}
By the induction hypothesis, it suffices to show that $(\alpha w)a_2(\alpha w){-1}\restr{T_{n-1}}=a_2\restr{T_{n-1}}$. By \cref{inductiona2}, 
\[ (\alpha w)a_2(\alpha w)^{-1}\restr{T_{n-1}}=\alpha [a_3,\alpha]a_2\alpha^{-1}\restr{T_{n-1}}=\alpha a_3 \alpha a_3^{-1} \alpha^{-1} a_2 \alpha^{-1}\restr{T_{n-1}} . \]

Since $\alpha \in U$,  the conjugate $a_3 \alpha a_3^{-1}$ is in $U$ and it commutes with $\alpha$ as $U$ is abelian. Hence, we have
 \[ 
 (\alpha w)a_2(\alpha w)^{-1}\restr{T_{n-1}}=a_3 \alpha a_3^{-1}  a_2 \alpha^{-1}\restr{T_{n-1}} = a_3 \alpha \gamma_2 \alpha^{-1}\restr{T_{n-1}}= a_3\gamma_2=a_2 \restr{T_{n-1}}.
 \]
 
 Now, we show that $wa_2w^{-1}\restr{T_n}=[a_3,\alpha]a_2\restr{T_n}$. Since $a_2=a_1a_3^{-1}$ and we showed that $wa_3w^{-1}\restr{T_n}=a_3\restr{T_n}$, 
 \[  
 wa_2w^{-1}\restr{T_n} =wa_1w^{-1}a_3^{-1}\restr{T_n}  =wa_1w^{-1}a_1^{-1}a_2\restr{T_n}=[w,a_1]a_2. 
 \]
  By our assumption on $w$, $[w,a_1]=[a_3,\alpha]$, hence we are done. 
 
 Next we calculate that for $\alpha=\gamma_i^2$, the property $[a_3,\alpha]=[w,a_1]$ holds. First of all 
 \[ [w,a_1]=wa_1w^{-1}a_1^{-1}=(\alpha, \alpha^{-1}). \]
 And for any $k,\ell \in \ZZ_2$ and $\alpha=\gamma_1^k\gamma_2^{\ell}$, we have \[ [a_3,\alpha]=a_3\alpha a_3^{-1}\alpha^{-1}=\gamma_1^{\ell-k}\gamma_2^{-k-\ell}, \]
 where we use \cref{conjugation by a3} for calculating the conjugates of $\gamma_1,\gamma_2$ by $a_3$.
 
 We calculate that for $k=2, \ell=0$, \[ [a_3,\gamma_1^2]=\gamma_1^{-2}\gamma_2^{-2}=(\gamma_1\gamma_2)^{-2}=(\gamma_1^{-1},\gamma_1)^{-2}. \]
 The last equation follows from \cref{gamma1gamma2}. Similarly for $k=0,\ell=2$,
 \[ [a_3,\gamma_2^2]=\gamma_1^{2}\gamma_2^{-2}=(\gamma_1\gamma_2^{-1})^{2}=[a_1,a_3^{-1}]^2. \]
We calculated in the proof of \cref{betas} that $[a_1,a_3^{-1}]=\beta_2^{-1}=(\gamma_2 ,\gamma_2^{-1}])$. This finishes the proof of the first and third statement. 

It is easy to calculate that $za_1z^{-1}=a_1$ since $z$ is diagonal. Assume for induction that $za_3z^{-1}\restr{T_{n-1}}=a_3^{-1}\restr{T_{n-1}}$, then 
 \begin{align*} za_3z^{-1}\restr{T_n}=&(za_3^{-1}a_2a_3z^{-1}\restr{T_{n-1}}, za_3z^{-1}\restr{T_{n-1}})  =(za_3^{-1}a_1z^{-1}\restr{T_{n-1}},  za_3z^{-1}\restr{T_{n-1}}) \\ =&(a_3a_1, a_3^{-1})\restr{T_n} =(a_2,a_3)^{-1}\restr{T_n}.
 \end{align*}
 This finishes the induction. Hence $za_1z^{-1}=a_1$ and $za_3z^{-1}=a_3^{-1}$. Now, $$z\gamma_1z^{-1}=za_2a_3^{-1}z^{-1}=za_1a_3^2z^{-1}=a_1a_3^{2}$$ since $a_3$ has order $4$. Similarly, 
 $$ z \gamma_2 z^{-1}=za_3^{-1}a_2z^{-1}=a_3z a_1 a_3^{-1}z^{-1}=a_3a_1a_3=a_2^{-1}a_3=\gamma_2^{-1}. $$
\end{proof}

\section{The quotient $N/G$ and $\ga$}\label{ga}

 Our ultimate goal in this section is to understand the constant field and the group $\ga$ of $f$.  We first study $N/U$ and hence $N/G$ since the constant field is determined by the quotient $\ga/G$ which is a subgroup of $N/G$.
 
 Since $U$ is normal in $N$, $N$ acts on $U$ by conjugation. Let $C_N(U)$ denote the centralizer of $U$ in $N$. This gives an injective homomorphism
 $ \phi: N/C_N(U) \to \Aut(U)$. Next, we show that $C_N(U)=U$.

\begin{proposition}\label{centralizer of U}
The centralizer of $U$ in $\Omega$ is equal to $U$ itself.
\end{proposition}
\begin{proof}
First note that $G$ acts transitively on level $n$ for every $n\geq 1$ and that any $g \in G$ is $\alpha a_3^k$ for some $\alpha \in U$ and $k \in \{0,1,2,3\}$. Let $v$ be the vertex of $T_n$ fixed by $a_3$. By the recursive structure of $a_3$ and the fact that $a_3 \restr{T_1}=\id$, such a vertex exists. Then the orbit of $v$ under $G$ equals its orbit under the $U$ action. Hence the $U$-orbit of $v$ is the set of all vertices on level $n$ and $U$ acts transitively on every level. By \cref{orderG} and \cref{index of U}, the order $U_n$ is $2^n$ for every $n\geq 3$. Hence, $U$ acts freely on each level for $n \geq 3$.

  Assume that $\gamma\in\Omega$ centralizes $U$. Fix $n\geq3$. Since $U$
 	acts transitively and freely on level $n$, there is a unique $b_n\in U$
 	such that
 	\[
 	b_n(v)=\gamma(v).
 	\]
 	Now let $u\in U$. Since $\gamma$ centralizes $U$ and $U$ is abelian,
 	\[
 	\gamma(u(v))
 	=u(\gamma(v))
 	=u(b_n(v))
 	=b_n(u(v)).
 	\]
 	Since every vertex on level $n$ is of the form $u(v)$ for some $u\in U$,
 	it follows that $\gamma$ and $b_n$ agree on level $n$. Thus, for every
 	$n\geq3$, there exists $b_n\in U$ such that
 	\[
 	\gamma|_{T_n}=b_n|_{T_n}.
 	\]
 	Hence $\gamma\in\overline{U}=U$, since $U$ is closed in $\Omega$.
 
\end{proof}

\subsection*{A description of the quotient $N/G$}
Let $\mu_4$ denote $G/U$. Then we have 
        $$ \mu_4 \hookrightarrow N/U \hookrightarrow \Aut(U).$$ 
Since $G$ is normal in $N$, $N/U$ lives in the normalizer of $\mu_4$ inside the automorphism group of $U$. This motivates us to study the automorphism group of $U$. 
By \cref{U torsion free}, $U$ is isomorphic to $\ZZ_2[i]$, hence it is a free $\ZZ_2$- module of rank two. 

We have the following $\ZZ_2$-module isomorphisms:
\begin{equation}
U \to \ZZ_2[i] \to \ZZ_2 \oplus \ZZ_2, \end{equation}

which sends  $\gamma_1 \mapsto i \mapsto (0,1)$ and $\gamma_2 \mapsto 1 \mapsto (1,0)$.

By \cref{conjugation by a3}, the action of $\mu_4$, which is generated by the image of $a_3$, on $U$ translates to multiplication by $i$ on $\ZZ_2[i]$ and taking $(a,b) \mapsto (-b ,a)$ on $\ZZ_2 \oplus \ZZ_2$. In other words, the isomorphism $U \to \ZZ_2[i]$ preserves the $\mu_4$-action on both sides.

 The automorphism group of $U$ is isomorphic to $\GL_2(\ZZ_2)$. We will identify $\Aut(U)$ by $\GL(2,\ZZ_2)$ using the explicit isomorphisms given above.

\begin{lemma}\label{normalizer of mu}\hfill
\begin{enumerate}
\item The image of $\mu_4$ in $\GL_2(\ZZ_2)$ is generated by $ \begin{pmatrix} 0 & -1 \\ 1 & 0 \end{pmatrix} $.
\item The centralizer of $\mu_4$ in $\GL_2(\ZZ_2)$ is 
\[ \{ \begin{pmatrix} a & -b \\ b & a \end{pmatrix} \in \GL_2(\ZZ_2) \; : \; a,b \in \ZZ_2 \}. \]
\item The normalizer of $\mu_4$ in $\GL_2(\ZZ_2)$ is generated by the centralizer of $\mu_4$ and the element $\begin{pmatrix} -1 & 0 \\ 0 & 1 \end{pmatrix} $.
Hence, it is isomorphic to $\ZZ_2[i]^{\times} \rtimes \langle \tau \rangle$, where 
$\tau $ is the generator of the Galois group of $\QQ_2(i)$ over $\QQ_2$.

\end{enumerate}
\end{lemma}

 \begin{proof}
Write $J=\begin{pmatrix} 0 & -1 \\ 1 & 0\end{pmatrix}$ and
$\tau=\begin{pmatrix} -1 & 0 \\ 0 & 1\end{pmatrix}$. Throughout, matrices act on
column vectors, so that $\gamma_2$ and $\gamma_1$ correspond to the first and
second standard basis vectors of $\ZZ_2\oplus\ZZ_2 $ respectively.

 By \cref{conjugation by a3}, $a_3$ sends $\gamma_2\mapsto\gamma_1$ and
$\gamma_1\mapsto\gamma_2^{-1}$, that is, $(1,0)\mapsto(0,1)$ and
$(0,1)\mapsto(-1,0)$. Hence, the image of $a_3$ is $J$, and $\mu_4$ is generated by
the image of $a_3$.

Let $A=\begin{pmatrix} a & b \\ c & d \end{pmatrix}\in\GL_2(\ZZ_2)$. Then
$AJ=JA$ if and only if $d=a$ and $c=-b$ hold, which describes
$C:=C_{\GL_2(\ZZ_2)}(\mu_4)$ since $\mu_4=\langle J\rangle$. The map
$a+bi\mapsto\begin{pmatrix} a & -b \\ b & a\end{pmatrix}$ induces an isomorphism
$C\to\ZZ_2[i]^{\times}$ since $\det\begin{pmatrix} a & -b \\ b & a\end{pmatrix}=a^2+b^2$ is a unit exactly
when $a+bi\in\ZZ_2[i]^{\times}$.

Since $\Aut(\mu_4)\simeq\ZZ/2\ZZ$, the natural map
$N_{\GL_2(\ZZ_2)}(\mu_4)\to\Aut(\mu_4)$ has kernel $C$, giving the exact sequence
\[ 1 \to C \to N_{\GL_2(\ZZ_2)}(\mu_4) \to \Aut(\mu_4) \to 1 .\]
It is surjective because $\tau J\tau^{-1}=\begin{pmatrix} 0 & 1 \\ -1 & 0\end{pmatrix}=J^{-1}$,
so $\tau$ normalizes $\mu_4$ without centralizing it; as $\tau^2=\id$, sending the
non-trivial automorphism of $\mu_4$ to $\tau$ splits the sequence. Note that
$\tau$ sends $\gamma_1\mapsto\gamma_1$ and $\gamma_2\mapsto\gamma_2^{-1}$
(compare \cref{xyz in N}). Finally,
\[ \tau\begin{pmatrix} a & -b \\ b & a\end{pmatrix}\tau^{-1}
   =\begin{pmatrix} a & b \\ -b & a\end{pmatrix}, \]
so under the isomorphism of $C\simeq \ZZ_2[i]^{\times} $ conjugation by $\tau$ sends $a+bi\mapsto a-bi$,
which is complex conjugation.
\end{proof}

By \cref{normalizer of mu}, we can view $N/U$ inside $\ZZ_2[i]^{\times} \rtimes \langle \tau \rangle$ under the isomorphism described in the proof of the lemma. In the next proposition, we identify $N_{\GL_2(\ZZ_2)}(\mu_4)$ by $\ZZ_2[i]^{\times} \rtimes \langle \tau \rangle$.

\begin{proposition}\label{action of xyz on U}
 The images of $x,y, z$ in $N_{\GL_2(\ZZ_2)}(\mu_4)$ respectively are $-1+2i, 3+2i$, and $\tau$.
\end{proposition}
\begin{proof}
By ~\cref{xyz in N}, $xa_3x^{-1}=a_3$ and $xa_2x^{-1}=[a_3,\gamma_1^2]a_2$. We compute using ~\cref{conjugation by a3} that 
\[ x\gamma_1x^{-1}=xa_2a_3^{-1}x^{-1}=[a_3,\gamma_1^2]a_2a_3^{-1}=\gamma_2^{-2}\gamma_1^{-2}\gamma_1=\gamma_2^{-2}\gamma_1^{-1}. \]

Similarly,

\[ y\gamma_1y^{-1}=ya_2a_3^{-1}y^{-1}=[a_3,\gamma_2^2]a_2a_3^{-1}=\gamma_1^{2}\gamma_2^{-2}\gamma_1=\gamma_2^{-2}\gamma_1^3. \]

By \cref{normalizer of mu}, it is not needed to compute the image of $\gamma_2$ under the action of $x$ and $y$ since $x$ and $y$ both commute with $a_3$. Hence the image of $x$ and $y$ are $\begin{pmatrix} -1 & -2 \\ 2 & -1 \end{pmatrix} $ and $\begin{pmatrix} 3 & -2 \\ 2 & 3 \end{pmatrix} $.

Again, by \cref{xyz in N}, $z$ fixes $\gamma_1$ and sends $\gamma_2$ to its own inverse. Hence, its image in $\GL_2(\ZZ_2)$ is $\begin{pmatrix} -1 & 0 \\ 0 & 1 \end{pmatrix} $. The result follows from the isomorphism given in the proof of \cref{normalizer of mu}.
\end{proof}

\begin{lemma}\label{N over U}
The elements $i,\ -1+2i,\ 3+2i$ topologically generate $\ZZ_2[i]^{\times}$. More precisely,
$$\ZZ_2[i]^{\times}= \langle i\rangle \oplus \ZZ_2(-1+2i)\oplus \ZZ_2(3+2i).$$
\end{lemma}
\begin{proof}

	Let $L=\QQ_2(i)$, with valuation ring $\ZZ_2[i]$ and uniformizer $\pi=1+i$.
	Then $L/\QQ_2$ is a totally ramified extension of degree $2$, with ramification
	index $e=2$ and residue field $\FF_2$. Since
	\[
	\pi^2=(1+i)^2=2i,
	\]
	we have $v(\pi)=1$ and hence $v(2)=2$. We write
	\[
	U^{(n)}=1+\pi^n\ZZ_2[i].
	\]

Since $q=\mid  \FF_2 \mid =2$, \cite[Ch.~II, Prop.~5.3]{Neukirch} gives $\ZZ_2[i]^{\times}=U^{(1)}$.
Each $U^{(n)}/U^{(n+1)}\simeq\FF_2$, so we have $[U^{(1)}:U^{(3)}]=4$. As $i-1=i\pi$ and $-1-1=-2$
have $v$ equal to $1$ and $2$, the group $\langle i\rangle$ (of order $4$) meets $U^{(3)}$
trivially, hence maps isomorphically onto $U^{(1)}/U^{(3)}$ and
\[ U^{(1)}=\langle i\rangle\times U^{(3)} . \]

Since $3>e/(p-1)=2$, where $p=2$ and $e=2$, \cite[Ch.~II, Prop.~5.5]{Neukirch} gives an isomorphism of
$\ZZ_2$-modules $$\log:U^{(3)}\xrightarrow{\ \sim\ }\pi^3\ZZ_2[i].$$ 

Composing this map with $a\pi^3\mapsto a$ gives an isomorphism $$\lambda:U^{(3)}\xrightarrow{\ \sim\ }\ZZ_2[i].$$ 
It therefore
suffices to show that $\lambda(-1+2i),\lambda(3+2i)$ generate $\ZZ_2[i]$ as a $\ZZ_2$-module.  Since $\pi^3=-2+2i$, we have
\[ u_1=-1+2i=1+\pi^3\in U^{(3)},\qquad u_2=3+2i=1-i\pi^3\in U^{(3)} . \]

In
\[
\log(1+x)=\sum_{n\ge1}(-1)^{n-1}\frac{x^n}{n},
\]
with \(v(x)=3\), we have
\[
v\!\left(\frac{x^n}{n}\right)-3
=3n-3-v(n)
=3n-3-2v_2(n).
\]
For \(n\ge3\), this is at least \(2\). Hence, modulo \(\pi^2\mathbb Z_2[i]\), only the terms with \(n=1,2\) contribute.
Using $\pi^3/2=i\pi$, we obtain
\[ \lambda(u_1)\equiv 1-i\pi,\qquad \lambda(u_2)\equiv -i+i\pi \pmod{\pi^2\ZZ_2[i]}. \]
Now $2\ZZ_2[i]=\pi^2\ZZ_2[i]$ and $\ZZ_2[i]/2\ZZ_2[i]=\FF_2\oplus\FF_2\pi$. Since
$i\equiv 1+\pi$ and $-i\equiv i \pmod{\pi^2}$, these residues are  respectively$1+\pi$ and $1$, which form an
$\FF_2$-basis of $\ZZ_2[i]/2\ZZ_2[i]$. As $\ZZ_2[i]$ is a finitely generated module over
the local ring $\ZZ_2$, Nakayama's lemma gives $\ZZ_2[i]=\ZZ_2\lambda(u_1)+\ZZ_2\lambda(u_2)$.  Thus, $u_1$ and $u_2$ topologically generate $U^{(3)}$.
\end{proof}

\begin{corollary}\label{N/U}
The quotient $N/U$ is isomorphic to the normalizer of $\mu_4$ in $\GL_2(\ZZ_2)$.
\end{corollary}
\begin{proof}
We know that $N/U$ is a subgroup of the normalizer of $\mu_4$ in $\GL_2(\ZZ_2)$, which is isomorphic to $\ZZ_2[i]^{\times} \rtimes \langle \tau \rangle$. We will identify $N/U$ with this group. By \cref{action of xyz on U}, $-1+2i, 3+2i$ and $\tau$ are in $N/U$. By \cref{N over U}, these elements topologically generate $\ZZ_2[i]^{\times} \rtimes \langle \tau \rangle$.
\end{proof}

\begin{theorem}\label{N over G}
The quotient $N/G$ is isomorphic to $U^{(3)}\rtimes\langle\tau\rangle$, where
$U^{(3)}=1+\pi^{3}\ZZ_2[i]$ is a free $\ZZ_2$-module of rank two and $\tau$ acts
by complex conjugation. 
\end{theorem}

\begin{proof}
By \cref{N/U} we have $N/U\simeq\ZZ_2[i]^{\times}\rtimes\langle\tau\rangle$, and
under this isomorphism $\mu_4=G/U$ corresponds to $\langle i\rangle$. Hence,
$$N/G\simeq\bigl(\ZZ_2[i]^{\times}\rtimes\langle\tau\rangle\bigr)/\langle i\rangle.$$
By \cref{N over U} we have $\ZZ_2[i]^{\times}=\langle i\rangle\times U^{(3)}$ and hence $\ZZ_2[i]^{\times}/\langle i\rangle\simeq U^{(3)}$. Since both factors are $\tau$-stable, it follows that $N/G\simeq U^{(3)}\rtimes\langle\tau\rangle$, where $\tau$ acts on $U^{(3)}$ by complex conjugation.

\end{proof}
 
 \begin{theorem}\label{abelianization of N}
 The maximal abelian quotient of $N/G$ is isomorphic to $\ZZ_2 \oplus \langle \tau \rangle$, which is isomorphic to $\ZZ_2^{\times}$.
 \end{theorem}

\begin{proof}
Write $A=\ZZ_2[i]^{\times}$. By \cref{N over U}, $A=\mu_4\times U^{(3)}\simeq \mu_4\times\ZZ_2\times\ZZ_2$. If $B$ is an abelian group with an action of $\langle\tau\rangle$, then 
$[\tau,b]=\tau(b)b^{-1}$ in $B\rtimes\langle\tau\rangle$ and in fact the commutator subgroup of $B \rtimes\langle\tau\rangle$ is 
$$(\tau-1)B:=\{\tau(b) b^{-1} \; | \; b \in B \}.$$

Consider the norm map $\operatorname{Nm} \colon A\to\ZZ_2^{\times}$. Since $\operatorname{Nm}\circ\tau=\operatorname{Nm}$, we have
$(\tau-1)A\subseteq\ker \operatorname{Nm}$, and $\operatorname{Nm}(i)=1$ gives $\mu_4\subseteq\ker \operatorname{Nm}$. Conversely,
let $u\in A$ with $\operatorname{Nm}(u)=1$. By Hilbert 90, there is $b\in \QQ_2(i)^{\times}$ with $u=\bar{b}b^{-1}=\tau(b)b^{-1}$.  Writing $b=\pi^{m}w$ with $w\in A$, and using
$\overline{\pi}\pi^{-1}=-i$, we obtain $u=(-i)^{m}\overline{w}w^{-1}$. Hence,
         \[ \ker \operatorname{Nm}=\mu_4\cdot(\tau-1)A. \]

As the image of $\operatorname{Nm}$ is $1+4\ZZ_2$, the map $(a,\tau^{\epsilon})\mapsto(\operatorname{Nm}(a),\tau^{\epsilon})$ is a surjection with
kernel $\mu_4\cdot(\tau-1)A$.  Thus,
\[ 1\to \mu_4\cdot(\tau-1)A\to A\rtimes\langle\tau\rangle
     \to(1+4\ZZ_2)\times\langle\tau\rangle\to1 . \]
Since $\tau(\mu_4)=\mu_4$ and the image of $\mu_4\cdot(\tau-1)A$ in $A/\mu_4$ is
$(\tau-1)(A/\mu_4)$, this descends to
\[ 1\to(\tau-1)(A/\mu_4)\to\bigl(A/\mu_4\bigr)\rtimes\langle\tau\rangle
     \to(1+4\ZZ_2)\times\langle\tau\rangle\to1 . \]
By \cref{N over G} we have $N/G\simeq\bigl(A/\mu_4\bigr)\rtimes\langle\tau\rangle$,
whose commutator subgroup is $(\tau-1)(A/\mu_4)$ by the first paragraph. Hence,
\[ (N/G)^{\mathrm{ab}}\simeq(1+4\ZZ_2)\times\langle\tau\rangle
   \simeq\ZZ_2\oplus\ZZ/2\ZZ\simeq\ZZ_2^{\times}, \]
where the last isomorphism follows from $\ZZ_2^{\times}=\{\pm1\}\times(1+4\ZZ_2)$.
\end{proof}

\begin{theorem}\label{N eq ga}
If $\zeta_8 $ is not contained in $k$, then $N=\ga$.
\end{theorem}
 \begin{proof}
By \cref{abelianization of N}, the maximal elementary abelian quotient of $N/G$ is isomorphic to $\ZZ/2\ZZ \oplus \ZZ/2\ZZ$. \cref{property of squareroottwo,property of i}  imply that if $\zeta_8 \not\in k$, then  $\ga_4/G_4=N_4/G_4 \simeq \ZZ/2\ZZ \oplus \ZZ/2\ZZ$. 

We will apply the so-called Frattini argument to $\ga/G$ and $N/G$. Denoting the Frattini subgroup of $N/G$ by $\Phi$, we have $(\ga/G) \Phi/\Phi=N/G$ and hence $(\ga/G) \Phi=N/G$. By  \cite[Proposition~6.2]{arithmeticbasilica}, we have $\ga/G=N/G$, and therefore $\ga=N$.

\end{proof}

\begin{lemma}\label{Qabelian}
Let $L/\QQ$ be a finite abelian extension in which only $2$ ramifies.
Then $L\subseteq\QQ(\mu_{2^{a}})$ for some $a$.
\end{lemma}

\begin{proof}
By the Kronecker--Weber theorem, $L\subseteq\QQ(\mu_n)$ for some $n$; choose such
an $n$ minimal. Suppose that an odd prime $p$ divides $n$, and write $n=p^{k}m$ with
$p\nmid m$.  Then, $\QQ(\mu_n)=\QQ(\mu_m)(\zeta_{p^{k}})$ and
$[\QQ(\mu_n):\QQ(\mu_m)]=\varphi(p^{k})$, where $\varphi$ denotes Euler's totient function.

Let $\mathfrak P$ be a prime of $\QQ(\mu_n)$ above $p$, lying over $\mathfrak p$
in $\QQ(\mu_m)$ and over $\mathfrak q$ in $\QQ(\zeta_{p^{k}})$. Since $p$ is
unramified in $\QQ(\mu_m)$ and totally ramified in $\QQ(\zeta_{p^{k}})$,
\[
  e(\mathfrak P/\mathfrak p)=e(\mathfrak P/p)
  =e(\mathfrak P/\mathfrak q)\,e(\mathfrak q/p)\geq\varphi(p^{k})
  =[\QQ(\mu_n):\QQ(\mu_m)].
\]
Hence, $\QQ(\mu_n)/\QQ(\mu_m)$ is totally ramified at $\mathfrak p$ and therefore
$\Gal(\QQ(\mu_n)/\QQ(\mu_m))$ is contained in the inertia group
$I_{\mathfrak P}$ of $\mathfrak P$ over $\QQ$.

As $p$ is unramified in $L$, the inertia group $I_{\mathfrak P}$ fixes $L$.
Hence $\Gal(\QQ(\mu_n)/\QQ(\mu_m))\subseteq\Gal(\QQ(\mu_n)/L)$, and therefore
$L\subseteq\QQ(\mu_m)$. This contradicts the minimality of $n$. Hence, $n$ is
a power of $2$.
\end{proof}

\begin{corollary}
Let $k=\QQ$ and let $F=K_{\infty}\cap\overline{\QQ}$ denote the constant field.
Then,
\[ F\cap\QQ^{\mathrm{ab}}=\QQ(\mu_{2^{\infty}}) . \]
In particular, $K_{\infty}$ contains all $2$-power roots of unity, and the
abelianization of $\Gal(F/\QQ)\simeq\ga/G$ is the $2$-adic cyclotomic character.
\end{corollary}

\begin{proof}
By \cref{discriminant}, the only prime that ramifies in $F$ is $2$. Every finite
subextension of $F\cap\QQ^{\mathrm{ab}}/\QQ$ is abelian and unramified outside
$2$, hence is contained in $\QQ(\mu_{2^{a}})$ for some $a$ by \cref{Qabelian}. Taking
the union over all such subextensions gives
$F\cap\QQ^{\mathrm{ab}}\subseteq\QQ(\mu_{2^{\infty}})$.

Since $F/\QQ$ is Galois with Galois group $\ga/G$ and
$F\cap\QQ^{\mathrm{ab}}$ is its maximal abelian subextension, we have
$\Gal(F\cap\QQ^{\mathrm{ab}}/\QQ)\simeq(\ga/G)^{\mathrm{ab}}\simeq\ZZ_2^{\times}$
by \cref{abelianization of N,N eq ga}. Restriction therefore induces a surjection
\[ A:=\Gal\bigl(\QQ(\mu_{2^{\infty}})/\QQ\bigr)\simeq\ZZ_2^{\times}
   \longrightarrow \Gal\bigl(F\cap\QQ^{\mathrm{ab}}/\QQ\bigr)\simeq\ZZ_2^{\times}, \]
with kernel $H$ which we claim is trivial. Write $A\simeq\ZZ/2\ZZ\times\ZZ_2$. The
group $H\cap(\{0\}\times\ZZ_2)$ is a closed subgroup of $\ZZ_2$, hence equals
$2^{n}\ZZ_2$ or $0$. In the first case, $A/H$ would be a quotient of
$\ZZ/2\ZZ\times\ZZ/2^{n}\ZZ$ and thus finite, which is impossible. Hence
$H\cap(\{0\}\times\ZZ_2)=0$, so $H$ embeds into $\ZZ/2\ZZ$. If $H$ had order $2$
it would be the torsion subgroup of $A$, since $\ZZ_2$ is torsion-free. Hence
$A/H \simeq\ZZ_2$ would be torsion-free, contradicting $A/H\simeq\ZZ_2^{\times}$.
Therefore, $H=\{\id\}$ and $F\cap\QQ^{\mathrm{ab}}=\QQ(\mu_{2^{\infty}})$.
\end{proof}


\bibliographystyle{plain}
\bibliography{refmon.bib}

@book{Neukirch,
  author    = {Neukirch, J{\"u}rgen},
  title     = {Algebraic Number Theory},
  series    = {Grundlehren der mathematischen Wissenschaften},
  volume    = {322},
  publisher = {Springer-Verlag},
  address   = {Berlin},
  year      = {1999},
  note      = {Translated from the 1992 German original by Norbert Schappacher},
  doi       = {10.1007/978-3-662-03983-0},
  isbn      = {3-540-65399-6}
}

@misc{Leung-Petsche,
      title={The Minkowski dimension of the image of an arboreal Galois representation}, 
      author={Chifan Leung and Clayton Petsche},
      year={2026},
      eprint={2512.18825},
      archivePrefix={arXiv},
      primaryClass={math.NT},
      url={https://arxiv.org/abs/2512.18825}, 
}

@article {Stoll,
	AUTHOR = {Stoll, Michael},
	TITLE = {Galois groups over {${\bf Q}$} of some iterated polynomials},
	JOURNAL = {Arch. Math. (Basel)},
	FJOURNAL = {Archiv der Mathematik},
	VOLUME = {59},
	YEAR = {1992},
	NUMBER = {3},
	PAGES = {239--244},
	ISSN = {0003-889X,1420-8938},
	MRCLASS = {12F10 (12E10 12F12)},
	MRNUMBER = {1174401},
	MRREVIEWER = {R.\ W. K. Odoni},
	DOI = {10.1007/BF01197321},
	URL = {https://doi.org/10.1007/BF01197321},
}

@misc{AH-img,
      title={Profinite Iterated Monodromy Groups of Unicritical Polynomials}, 
      author={Ophelia Adams and Trevor Hyde},
      year={2025},
      eprint={2504.13028},
      archivePrefix={arXiv},
      primaryClass={math.NT},
      url={https://arxiv.org/abs/2504.13028}, 
}

@article {foz,
    AUTHOR = {Ferraguti, Andrea and Ostafe, Alina and Zannier, Umberto},
     TITLE = {Cyclotomic and abelian points in backward orbits of rational
              functions},
   JOURNAL = {Adv. Math.},
  FJOURNAL = {Advances in Mathematics},
    VOLUME = {438},
      YEAR = {2024},
     PAGES = {109463},
      ISSN = {0001-8708,1090-2082},
   MRCLASS = {37P05 (11R18 11R20 37P15)},
  MRNUMBER = {4686319},
       DOI = {10.1016/j.aim.2023.109463},
       URL = {https://doi.org/10.1016/j.aim.2023.109463},
}

@misc{Pinkpolyn,
      title={Profinite iterated monodromy groups arising from quadratic polynomials}, 
      author={Richard Pink},
      year={2013},
      eprint={1307.5678},
      archivePrefix={arXiv},
      primaryClass={math.GR}
}

@misc{Pinkrational,
      title={Profinite iterated monodromy groups arising from quadratic morphisms with infinite postcritical orbits}, 
      author={Richard Pink},
      year={2013},
      eprint={1309.5804},
      archivePrefix={arXiv},
      primaryClass={math.GR}
}

@article {EKO,
    AUTHOR = {Ejder, \"Ozlem and Kara, Yasemin and Ozman, Ekin},
     TITLE = {Iterated monodromy group of a {PCF} quadratic non-polynomial
              map},
   JOURNAL = {Manuscripta Math.},
  FJOURNAL = {Manuscripta Mathematica},
    VOLUME = {175},
      YEAR = {2024},
    NUMBER = {1-2},
     PAGES = {561--590},
      ISSN = {0025-2611,1432-1785},
   MRCLASS = {11G32 (12F10 37P05 37P15)},
  MRNUMBER = {4790572},
MRREVIEWER = {Andrea\ Ferraguti},
       DOI = {10.1007/s00229-024-01549-z},
       URL = {https://doi.org/10.1007/s00229-024-01549-z},
}

@article {Odoni85,
    AUTHOR = {Odoni, R. W. K.},
     TITLE = {The {G}alois theory of iterates and composites of polynomials},
   JOURNAL = {Proc. London Math. Soc. (3)},
  FJOURNAL = {Proceedings of the London Mathematical Society. Third Series},
    VOLUME = {51},
      YEAR = {1985},
    NUMBER = {3},
     PAGES = {385--414},
      ISSN = {0024-6115},
   MRCLASS = {12E10 (11B37 30F10)},
  MRNUMBER = {805714},
MRREVIEWER = {K. Kiyek},
       DOI = {10.1112/plms/s3-51.3.385},
       URL = {https://doi.org/10.1112/plms/s3-51.3.385},
}

@article {Arborealcubic,
    AUTHOR = {Benedetto, Robert L. and Faber, Xander and Hutz, Benjamin and
              Juul, Jamie and Yasufuku, Yu},
     TITLE = {A large arboreal {G}alois representation for a cubic
              postcritically finite polynomial},
   JOURNAL = {Res. Number Theory},
  FJOURNAL = {Research in Number Theory},
    VOLUME = {3},
      YEAR = {2017},
     PAGES = {Art. 29, 21},
      ISSN = {2363-9555},
}

@article {BEK,
    AUTHOR = {Bouw, Irene I. and Ejder, \"{O}zlem and Karemaker, Valentijn},
     TITLE = {Dynamical {B}elyi maps and arboreal {G}alois groups},
   JOURNAL = {Manuscripta Math.},
  FJOURNAL = {Manuscripta Mathematica},
    VOLUME = {165},
      YEAR = {2021},
    NUMBER = {1-2},
     PAGES = {1--34},
      ISSN = {0025-2611},
   MRCLASS = {11G32 (11R32 12F10 37P05 37P15)},
  MRNUMBER = {4242559},
       DOI = {10.1007/s00229-020-01204-3},
       URL = {https://doi.org/10.1007/s00229-020-01204-3},
}

@incollection {Jonessurvey,
    AUTHOR = {Jones, Rafe},
     TITLE = {Galois representations from pre-image trees: an arboreal
              survey},
 BOOKTITLE = {Actes de la {C}onf\'{e}rence ``{T}h\'{e}orie des {N}ombres et
              {A}pplications''},
    SERIES = {Publ. Math. Besan\c{c}on Alg\`ebre Th\'{e}orie Nr.},
    VOLUME = {2013},
     PAGES = {107--136},
 PUBLISHER = {Presses Univ. Franche-Comt\'{e}, Besan\c{c}on},
      YEAR = {2013},
   MRCLASS = {11R32 (11F80 37P15)},
  MRNUMBER = {3220023},
MRREVIEWER = {Thomas Ward},
}

@article {FP20,
    AUTHOR = {Ferraguti, Andrea and Pagano, Carlo},
     TITLE = {Constraining images of quadratic arboreal representations},
   JOURNAL = {Int. Math. Res. Not. IMRN},
  FJOURNAL = {International Mathematics Research Notices. IMRN},
      YEAR = {2020},
    NUMBER = {22},
     PAGES = {8486--8510},
      ISSN = {1073-7928},
   MRCLASS = {11R32 (37P05)},
  MRNUMBER = {4216695},
MRREVIEWER = {Joseph H. Silverman},
       DOI = {10.1093/imrn/rnaa243},
       URL = {https://doi.org/10.1093/imrn/rnaa243},
}

@article {Jones-Manes,
    AUTHOR = {Jones, Rafe and Manes, Michelle},
     TITLE = {Galois theory of quadratic rational functions},
   JOURNAL = {Comment. Math. Helv.},
  FJOURNAL = {Commentarii Mathematici Helvetici. A Journal of the Swiss
              Mathematical Society},
    VOLUME = {89},
      YEAR = {2014},
    NUMBER = {1},
     PAGES = {173--213},
      ISSN = {0010-2571},
   MRCLASS = {37P15 (11R32)},
  MRNUMBER = {3177912},
MRREVIEWER = {Andreas O. Bender},
       DOI = {10.4171/CMH/316},
       URL = {https://doi.org/10.4171/CMH/316},
}

@article {arithmeticbasilica,
    AUTHOR = {Ahmad, Faseeh and Benedetto, Robert L. and Cain, Jennifer and
              Carroll, Gregory and Fang, Lily},
     TITLE = {The arithmetic basilica: a quadratic {PCF} arboreal {G}alois
              group},
   JOURNAL = {J. Number Theory},
  FJOURNAL = {Journal of Number Theory},
    VOLUME = {238},
      YEAR = {2022},
     PAGES = {842--868},
      ISSN = {0022-314X},
   MRCLASS = {37P05 (11R32 14G25)},
  MRNUMBER = {4430121},
       DOI = {10.1016/j.jnt.2021.10.004},
       URL = {https://doi.org/10.1016/j.jnt.2021.10.004},
}

@article {CH,
	AUTHOR = {Cullinan, John and Hajir, Farshid},
	TITLE = {Ramification in iterated towers for rational functions},
	JOURNAL = {Manuscripta Math.},
	FJOURNAL = {Manuscripta Mathematica},
	VOLUME = {137},
	YEAR = {2012},
	NUMBER = {3-4},
	PAGES = {273--286},
	ISSN = {0025-2611},
	MRCLASS = {37P05 (11R29)},
	MRNUMBER = {2875279},
	MRREVIEWER = {Joseph H. Silverman},
	DOI = {10.1007/s00229-011-0460-y},
	URL = {https://doi-org.ezproxy.lib.utexas.edu/10.1007/s00229-011-0460-y},
}

@incollection {BelyiEjder,
    AUTHOR = {Ejder, \"{O}zlem},
     TITLE = {Arithmetic monodromy groups of dynamical {B}elyi maps},
 BOOKTITLE = {Arithmetic, geometry, cryptography, and coding theory 2021},
    SERIES = {Contemp. Math.},
    VOLUME = {779},
     PAGES = {91--102},
 PUBLISHER = {Amer. Math. Soc., [Providence], RI},
      YEAR = {[2022] \copyright 2022},
   MRCLASS = {11G32 (12F10 37P05 37P15)},
  MRNUMBER = {4445772},
       DOI = {10.1090/conm/779/15677},
       URL = {https://doi.org/10.1090/conm/779/15677},
}

@article {Odoni2,
    AUTHOR = {Odoni, R. W. K.},
     TITLE = {On the prime divisors of the sequence {$w_{n+1}=1+w_1\cdots w_n$}},
   JOURNAL = {J. London Math. Soc. (2)},
  FJOURNAL = {Journal of the London Mathematical Society. Second Series},
    VOLUME = {32},
      YEAR = {1985},
    NUMBER = {1},
     PAGES = {1--11},
      ISSN = {0024-6107},
   MRCLASS = {11N37 (11B37)},
  MRNUMBER = {813379},
MRREVIEWER = {W. Narkiewicz},
       DOI = {10.1112/jlms/s2-32.1.1},
       URL = {https://doi.org/10.1112/jlms/s2-32.1.1},
}

@article {Odoni88,
    AUTHOR = {Odoni, R. W. K.},
     TITLE = {Realising wreath products of cyclic groups as {G}alois groups},
   JOURNAL = {Mathematika},
  FJOURNAL = {Mathematika. A Journal of Pure and Applied Mathematics},
    VOLUME = {35},
      YEAR = {1988},
    NUMBER = {1},
     PAGES = {101--113},
      ISSN = {0025-5793},
   MRCLASS = {12F10 (20E22)},
  MRNUMBER = {962740},
MRREVIEWER = {Antonio Jos\'{e} Engler},
       DOI = {10.1112/S002557930000632X},
       URL = {https://doi.org/10.1112/S002557930000632X},
}

@article {BGJT25s,
    AUTHOR = {Benedetto, Robert L. and Ghioca, Dragos and Juul, Jamie and
              Tucker, Thomas J.},
     TITLE = {Specializations of iterated {G}alois groups of {PCF} rational
              functions},
   JOURNAL = {Math. Ann.},
  FJOURNAL = {Mathematische Annalen},
    VOLUME = {392},
      YEAR = {2025},
    NUMBER = {1},
     PAGES = {1031--1050},
      ISSN = {0025-5831,1432-1807},
       DOI = {10.1007/s00208-025-03110-z},
       URL = {https://doi.org/10.1007/s00208-025-03110-z},
}

@article {Ejder26,
    AUTHOR = {Ejder, \"Ozlem},
     TITLE = {Galois theory of quadratic rational functions with periodic
              critical points},
   JOURNAL = {J. Number Theory},
  FJOURNAL = {Journal of Number Theory},
    VOLUME = {280},
      YEAR = {2026},
     PAGES = {212--245},
      ISSN = {0022-314X,1096-1658},
   MRCLASS = {37P05 (11R32 20E08 37P15 37P25)},
  MRNUMBER = {4959722},
       DOI = {10.1016/j.jnt.2025.08.010},
       URL = {https://doi.org/10.1016/j.jnt.2025.08.010},
}

@article {BHL,
    AUTHOR = {Bush, Michael R. and Hindes, Wade and Looper, Nicole R.},
     TITLE = {Galois groups of iterates of some unicritical polynomials},
   JOURNAL = {Acta Arith.},
  FJOURNAL = {Acta Arithmetica},
    VOLUME = {181},
      YEAR = {2017},
    NUMBER = {1},
     PAGES = {57--73},
      ISSN = {0065-1036,1730-6264},
   MRCLASS = {11R32 (14G05 37P15)},
  MRNUMBER = {3720002},
MRREVIEWER = {Nicole\ Sutherland},
       DOI = {10.4064/aa8599-8-2017},
       URL = {https://doi.org/10.4064/aa8599-8-2017},
}

@article {BDcolliding,
    AUTHOR = {Benedetto, Robert L. and Dietrich, Anna},
     TITLE = {Arboreal {G}alois groups for quadratic rational functions with
              colliding critical points},
   JOURNAL = {Math. Z.},
  FJOURNAL = {Mathematische Zeitschrift},
    VOLUME = {308},
      YEAR = {2024},
    NUMBER = {1},
     PAGES = {Paper No. 7, 33},
      ISSN = {0025-5874,1432-1823},
   MRCLASS = {37P05 (11R32 14G25)},
  MRNUMBER = {4780035},
       DOI = {10.1007/s00209-024-03566-w},
       URL = {https://doi.org/10.1007/s00209-024-03566-w},
}

\end{document}